\documentclass[11pt]{amsart}
\usepackage{hyperref}
\usepackage{tabularx,booktabs}
\usepackage{caption}
\usepackage{cite}
\usepackage{lipsum}
\usepackage{amsfonts,amssymb,tikz,amsmath}
\usepackage{graphicx}
\usepackage{epstopdf}
\usepackage[all,import]{xy}
\usepackage{multirow}
\usepackage{tikz-cd}
\usetikzlibrary{matrix, calc, arrows}
\usepackage[normalem]{ulem}
\usepackage{smartdiagram}
\useunder{\uline}{\ul}{}
\usetikzlibrary{arrows.meta}
\usetikzlibrary{er,positioning}
\usepackage{adjustbox}
\usepackage{blindtext}
\usepackage{etex}
\usepackage{amsmath}
\usepackage{amsfonts}
\usepackage{accents,cases}
\usepackage{algorithm} 
\usepackage{algorithmicx}
\usepackage{algpseudocode} 
\usepackage{graphicx} 
\usepackage{epstopdf}
\usepackage{amscd}
\usepackage{amsthm}
\usepackage{amssymb}
\usepackage{bm}
\usepackage{latexsym,float}
\usepackage{makecell}
\usepackage{multicol, multirow}
\usepackage{euscript}
\usepackage{epsfig}

\usepackage{tikz}
\usetikzlibrary{shapes,arrows}

\usepackage{graphics}
\usepackage{array}
\usepackage{enumerate}
\usepackage{color,tikz}

\usepackage{algorithm,algpseudocode,float}
\usepackage{lipsum}



\usetikzlibrary{arrows,chains,matrix,positioning,scopes}
\makeatletter
\tikzset{join/.code=\tikzset{after node path={%
			\ifx\tikzchainprevious\pgfutil@empty\else(\tikzchainprevious)%
			edge[every join]#1(\tikzchaincurrent)\fi}}}
\makeatother
\tikzset{>=stealth',every on chain/.append style={join},
	every join/.style={->}}
\tikzstyle{labeled}=[execute at begin node=$\scriptstyle,
execute at end node=$]
\usepackage{wasysym}
\usepackage{hyperref}
\usepackage{pdfsync}
\usepackage{comment}
\usepackage[all]{xy}

\allowdisplaybreaks[4]

\renewcommand{\vec}[1]{\mbox{\boldmath$#1$}}

\newcommand{\bel}[1]{\begin{equation}\label{#1}}
	
	\newcommand{\be}{\begin{equation}}

		\newcommand{\sgn}{\ensuremath{\mathrm{Sgn}}}

		\newcommand{\ie}{\ensuremath{\mathrm{i.e.,}}}

		\newcommand{\vol}{\ensuremath{\mathrm{vol}}}

			\newcommand{\diag}{\ensuremath{\mathrm{diag}}}

		\newcommand{\Hmm}[1]{\leavevmode{\marginpar{\tiny%
					$\hbox to 0mm{\hspace*{-0.5mm}$\leftarrow$\hss}%
					\vcenter{\vrule depth 0.1mm height 0.1mm width \the\marginparwidth}%
					\hbox to
					0mm{\hss$\rightarrow$\hspace*{-0.5mm}}$\\\relax\raggedright #1}}}
		
		\newtheorem{theorem}{Theorem}[section]

		\newtheorem{lemma}[theorem]{Lemma}
		\newtheorem{corollary}[theorem]{Corollary}
		\newtheorem{definition}[theorem]{Definition}
		
		\newtheorem{remark}[theorem]{Remark}
		
		\newtheorem{proposition}[theorem]{Proposition}

		\DeclareMathOperator*{\argmin}{\ensuremath{arg\,min}}

\begin{document}
			\title{Dual Cheeger Constants, Signless 1-Laplacians and Maxcut}
			
			\author{Sihong Shao}
\email{sihong@math.pku.edu.cn}

\address{CAPT, LMAM and School of Mathematical Sciences,  Peking University, Beijing 100871, China}

			\author{Chuan Yang}
			\email{chuanyang@pku.edu.cn}
			
			\address{LMAM and School of Mathematical Sciences,  Peking University, Beijing 100871, China}
			
			\author{Dong Zhang}
			\email{dongzhang@math.pku.edu.cn}
			
			\address{LMAM and School of Mathematical Sciences,  Peking University, Beijing 100871, China}

			\begin{abstract}
				The first nontrivial lower bound of the worst-case approximation ratio for the maxcut problem was achieved via the dual Cheeger problem, whose optimal value is referred to the dual Cheeger constant $h^+$, and later improved through its modification $\widehat{h}^+$. However, the dual Cheeger problem and its modification themselves are relatively unexplored, especially lack of effective approximate algorithms. To this end, we first derive equivalent spectral formulations of $h^+$ and  $\widehat{h}^+$ within the framework of the nonlinear spectral theory of signless 1-Laplacian, present their interactions with the Laplacian matrix and {1-Laplacian}, and then use them to develop an inverse power algorithm that leverages the local linearity of the objective functions involved. We prove that the inverse power algorithm monotonically converges to a ternary-valued eigenvector, and provide the approximate values of $h^+$ and  $\widehat{h}^+$ on G-set for the first time. The recursive spectral cut algorithm for the maxcut problem can be enhanced by integrating into the inverse power algorithms, leading to significantly improved approximate values on G-set. Finally, we show that the lower bound of the worst-case approximation ratio for the maxcut problem within the recursive spectral cut framework can not be improved beyond \textcolor{black}{$0.769$}. 
			\end{abstract}
			\maketitle
			\makeatletter
			\@addtoreset{equation}{section}
			\makeatother   
			
			
			
			
			

			\section{Introduction}

			Given an unweighted and undirected graph $G=(V,E)$ with vertex set $V=\{1,2,\dots,n\}$ and edge set $E$,
			the Cheeger constant  (i.e., the Cheeger problem)
			\begin{equation}\label{eq:Cheeger0}
				h(G)=\min\limits_{S\subset V,S\not\in\{\emptyset,V\}} \frac{|E( S,S^c)|}{\min\{\vol(S),\vol(S^c)\}}
			\end{equation}
			is often adopted to calculate the conductance \cite{Alon1986}, where $E(V_1,V_2)$ collects the set of edges that cross $V_1$ and $V_2$, 
			$|E(V_1,V_2)|$ gives its amount, $\vol(V_1)=\sum_{i\in V_1} d_i$ with $d_i$ being the degree of the vertex $i$,
			and $S^c=V\backslash S$ denotes the complement of $S$ in $V$. 
			After introducing an indicative vector $\vec x=(x_1, x_2, \dots, x_n)\in\mathbb{Z}_2^n=\{0,1\}^n${\color{black}, where} $x_i$ equals  $1$ if $i\in S$, and $0$ if $i\in S^c$, 
			namely, $\vec x = \vec 1_S$, we are able to rewrite $h(G)$ into 
			\begin{equation}\label{eq:Cheeger1}
				h(G)=\min_{\vec x\in\mathbb{Z}_2^n\backslash\{\vec 0\},\|\vec x\|_{p,d}\leq \frac{1}{2}\vol(V)}\frac{I_p(\vec x)}{\|\vec x\|_{p,d}}, \quad p\in\{1, 2\},
			\end{equation}
			where $I_p(\vec x)=\sum_{\{i,j\}\in E}|x_i-x_j|^p$ and $\|\vec x\|_{p,d}=\sum_{i\in V} d_i|x_i|^p$.
			In view of the fact that $I_p(\vec x)$ and $\|\vec x\|_{p,d}$ are both homogeneous functions of degree $p$,
			one turns to search for the minimizers of $I_p(\vec x)$ in $X_p=\{\vec x\in \mathbb{R}^n: \|\vec x\|_{p,d}=1\}$
			after an often-used relaxation of $\vec x$ from $\mathbb{Z}_2^n$ to $\mathbb{R}^n$.
			In the case of $p=2$, using the standard Lagrange multiplier technique yields the $\Delta_2$-eigenproblem
			\begin{equation}\label{egL2}
				\Delta_2\vec x=\lambda D \vec x, 
			\end{equation}
			where $\Delta_2 \vec x=\frac12 \partial I_2(\vec x)$, $D\vec x = \frac12 \partial \|\vec x\|_{2,d}$ and $D=\diag{(d_1,d_2,\dots,d_n)}$. 
			If all the eigenvalues are arranged in an ascending order: $0=\lambda_1\le \lambda_2\le\cdots\le \lambda_n\leq 2$,
			then we have the famous Cheeger inequality \cite{Chung1997}, 
			\begin{equation}\label{ieq:che0}
				\frac{\lambda_2}{2}\le h(G) \le \sqrt{2\lambda_2},
			\end{equation}
			which adopts $\lambda_2$, the second smallest $\Delta_2$-eigenvalue,  to bound $h(G)$. 
			Usually, $\Delta_2$ has been also known as the so-called the Laplacian matrix or the graph Laplacian. As a contrast, for the case of $p=1$,  using a similar but formal Lagrange multiplier technique, one may arrive at the $\Delta_1$-eigenproblem \cite{HeinBuhler2010,SzlamBresson2010,Chang2015}
			\begin{equation}\label{eq:L1}
				\vec 0\in \Delta_1 \vec x - \mu D \sgn(\vec x)\;\;(\text{or }\Delta_1 \vec x \cap \mu D \sgn(\vec x)\neq \emptyset),
			\end{equation}
			where $\Delta_1 \vec x=\partial I_1(\vec x)$ and $D\sgn(\vec x)=\partial \|\vec x\|_{1,d}$ still hold formally, albeit with
			the subgradient ``$\partial$" and set-valued sign function $\sgn(t)$ that returns $\{1\}$ for $t>0$, $\{-1\}$ for $t<0$, and $[-1,1]$ for $t=0$. Compared with the inequality relation shown in {Equation~}\eqref{ieq:che0} for $p=2$, a remarkable equality relation was then digged out \cite{Chang2015,ChangShaoZhang2015}
			\begin{equation}\label{eq:cheeger}
				h(G)=\mu_2,
			\end{equation}
			where $\mu_2$ gives the second smallest $\Delta_1$-eigenvalue. More importantly, 
			the corresponding $\Delta_1$-eigenvectors provide exact Cheeger cuts \cite{Chang2015} where $\Delta_1$ was called the 1-Laplacian. {\color{black}Figure~}\ref{fig::eigenproblem1} cartoons the relations among $h(G)$, $\mu_2$ and $\lambda_2$. Several dedicated algorithms have been developed for approximating $h(G)$ through corresponding eigenvectors \cite{Luxburg2007,ChangShaoZhang2015},
			where the biggest strength of $\Delta_1$-eigenproblem over $\Delta_2$-eigenproblem is that 
			the characteristic function on any nodal domain of an eigenvector of the former is again an eigenvector with the same eigenvalue and thus it does not need any rounding to produce a binary cut \cite{ChangShaoZhang2017-am}, whereas a threshold rounding step is necessary to convert an eigenvector of the latter from $\mathbb{R}^n$ back to $\mathbb{Z}_2^n$.
			
			\tikzset{%
				>={Latex[width=2mm,length=2mm]},
				base/.style = {rectangle, draw=black,
					minimum width=2cm, minimum height=1cm,
					text centered},
				activityStarts/.style = {base},
				startstop/.style = {base},
				activityRuns/.style = {ellipse,draw=black,text centered, minimum height=1cm, inner xsep=-5pt},
				process/.style = {base, minimum width=2.5cm},
			}

			\begin{figure}[htbp]
				\centering
				\begin{tikzpicture}[scale=1,auto,node distance=2cm]
					
					\node[startstop] (node1)[align = center] {$\Delta_1$-eigenproblem};
					\node[startstop] (rel1) [right = of node1, align = center] {Cheeger problem};
					\node[activityRuns] (node2) [right = of rel1, align = center] {$\Delta_2$-eigenproblem};
					\path[<->] (node1) edge node {$p=1$} (rel1);
					\path[<->] (rel1) edge node {equivalent} (node1);
					\path[-] (rel1) edge node {$p=2$}(node2);
					\path[<-] (node2) edge node {relaxed}(rel1);
				\end{tikzpicture}
				\caption{The Cheeger problem given in {\color{black}Equation~}\eqref{eq:Cheeger0} or \eqref{eq:Cheeger1} and its relations with the $\Delta_1$- and $\Delta_2$-eigenproblems shown in {\color{black}Equations~}\eqref{egL2} and \eqref{eq:L1}, respectively.}
				\label{fig::eigenproblem1}
			\end{figure}
			
			This paper continues the study and aims to leverage the equivalence as displayed in {\color{black}Figure~}\ref{fig::eigenproblem1} to investigate the graph cut problems
			\begin{align}
				h^+(G)&=\max\limits_{V_1\cap V_2=\emptyset,V_1\cup V_2\ne \emptyset} \frac{2|E(V_1,V_2)|}{\vol(V_1\cup V_2)},\label{eq:dual-Cheeger-constant}\\
				\widehat{h}^+(G)&=\max\limits_{V_1\cap V_2=\emptyset,V_1\cup V_2\ne \emptyset} \frac{2|E(V_1,V_2)|+|E\left(V_1\cup V_2,(V_1\cup V_2)^c\right)|}{\vol(V_1\cup V_2)+|E\left(V_1\cup V_2,(V_1\cup V_2)^c\right)|}.\label{eq:modified-dual-Cheeger-constant}
			\end{align}
			Here $h^+(G)$ was named the dual Cheeger constant in respect to a ``dual" version of {\color{black}Equation~}\eqref{ieq:che0} \cite{Trevisan2012,BauerJost2013}
			\begin{equation}\label{eq:dual-cheeger-inequality}
				1-\sqrt{\lambda_n(2-\lambda_n)}\le h^+(G) \le \frac{\lambda_n}{2},
			\end{equation}
			which adopts $\lambda_n$, the largest $\Delta_2$-eigenvalue,  to bound $h^+(G)$. 
			Based on the dual Cheeger problem \eqref{eq:dual-Cheeger-constant}, 
			a recursive spectral cut (RSC) algorithm
			was proposed to approximate the maxcut problem
			\begin{equation}\label{eq:maxcut0}
				h_{\max}(G) =\max\limits_{V_1\subset V} \frac{2|E(V_1,V_1^c)|}{\vol(V)},
			\end{equation}
			and achieved an approximation worst-case ratio of at least $0.531$ \cite{Trevisan2012}.
			This represents a positive milestone for approximating $h_{\max}(G)$ since 
			it gives the first approximation ratio of strictly larger than 1/2 that is not based on the semidefinite programming (SDP) relaxation. Within the same RSC framework, this ratio was later improved to $0.614$ after replacing 
			$h^+(G)$  with $\widehat{h}^+(G)$ because the latter raises the importance of the edges crossing the ``uncut part'' $(V_1\cup V_2)^c$ \cite{Soto2015}. In fact, a refined inequality connecting $h^+(G)$, $\lambda_n$, $h_{\max}(G)$ and $\widehat{h}^+(G)$ can be readily derived.
			
			\begin{theorem}
				\label{th:1-ieq}
				We have
				\begin{equation}\label{eq:relation-inequality}
					\frac12\le h_{\max}(G)\le h^+(G)\le \widehat{h}^+(G)\le \frac{\lambda_n-h^+(G)}{1+\lambda_n-2h^+(G)} \le \frac{\lambda_n-\widehat{h}^+(G)}{1+\lambda_n-2\widehat{h}^+(G)} \le \frac{\lambda_n}{2}\le1.
				\end{equation}
			\end{theorem}
			
			Theorem~\ref{th:1-ieq} implies that the dual Cheeger inequality \eqref{eq:dual-cheeger-inequality} still holds when we replace $h^+(G)$ with $\widehat{h}^+(G)$. In this sense, we call $\widehat{h}^+(G)$ the modified dual Cheeger constant and {\color{black}Equation~}\eqref{eq:modified-dual-Cheeger-constant} the modified dual Cheeger problem. 
			For the sake of brevity, we denote $I(\vec x)=I_1(\vec x)$, $\|\vec x\|=\|\vec x\|_{1,d}$ and $X=X_1$ hereafter. Let $\mathbb{Z}_3^n=\{-1,0,1\}^n$, we are able to rewrite $h^+(G)$ and $\widehat{h}^+(G)$ into
			\begin{align} 
				1-h^+(G) &=\min\limits_{\vec x\in\mathbb{Z}_3^n\backslash\{\vec 0\}}\frac{I^+(\vec x)}{\|\vec x\|}, \label{eq:dc-discrete} \\
				1-\widehat{h}^+(G) &=\min\limits_{\vec x\in\mathbb{Z}_3^n\backslash\{\vec 0\}}\frac{I^+(\vec x)}{I^+(\vec x)+I(\vec x)}, \label{eq:mdc-discrete}
			\end{align}
			where $I^+(\vec x)=\sum_{\{i,j\}\in E}|x_i+x_j|$.  Analogous to the  $\Delta_1$-eigenproblem \eqref{eq:L1} for $h(G)$, we accordingly formulate a $\Delta_1^+$-eigenproblem with $(\mu^+,\vec x)\in\mathbb{R}^1\times {X}$ for $h^+(G)$: 
			\begin{equation}
				\vec 0 \in \Delta_1^+ \vec x - \mu^+ D \sgn(\vec x)\;\;(\text{or }\Delta_1^+ \vec x\cap \mu^+ D \sgn(\vec x)\neq \emptyset),\label{eq:L1-plus}
			\end{equation}
			with $\Delta_1^+\vec x=\partial I^+(\vec x)$, as well as a $\widehat{\Delta}_1^+$-eigenproblem with $(\widehat{\mu}^+,\vec x)\in\mathbb{R}^1\times {X}$ for $\widehat{h}^+(G)$:
			\begin{equation}
				\vec{0} \in (1-\widehat{\mu}^+)\Delta_{1}^{+} \vec{x}-\widehat{\mu}^+ \Delta_1\vec x \;\;\left(\text {or }\, \widehat{\mu}^+ \Delta_1\vec x \cap (1-\widehat{\mu}^+)\Delta_{1}^{+} \vec x\neq \emptyset\right),
				\label{eq:modified-L1-plus}
			\end{equation}
			where $\widehat{\Delta}_1^+$ can be regarded as a linear combination of $\Delta_1$ and $\Delta_1^+$. 
			Then, we are still able to establish the equality relations between the dual Cheeger constants and their corresponding eigenproblems.
			
			\begin{theorem}
				\label{th:1-eigen}
				\begin{align}
					1- h^+(G)&=\mu^+_1=\min\limits_{\vec  x\in\mathbb{R}^n\backslash\{\vec 0\}}\frac{I^+(\vec  x)}{\|\vec  x\|},\label{eq:min}\\
					1- \widehat{h}^+(G)&=\widehat{\mu}^+_1=\min\limits_{\vec  x\in\mathbb{R}^n\backslash\{\vec 0\}}\frac{I^+(\vec  x)}{I^+(\vec x)+I(\vec x)},\label{eq:modified-min}
				\end{align}
				where $\mu_1^+$ and $\widehat{\mu}_1^+$ give the smallest eigenvalues of $\Delta_1^+$-eigenproblem \eqref{eq:L1-plus} and $\widehat{\Delta}_1^+$-eigenproblem \eqref{eq:modified-L1-plus}, respectively.
			\end{theorem}
			
			We call $\Delta_1^+$ the signless 1-Laplacian and $\widehat{\Delta}_1^+$ the modified signless 1-Laplacian. In a word, our effort will make the research-network displayed in {\color{black}Figure~}\ref{fig:relations} more systematic, productive and well developed. In particular, we obtain new connections among these objects rooted in the discrete graph cut problems and the continuous eigenproblems via the equivalent spectral graph theory (see the upper triangular part of {\color{black}Figure~}\ref{fig:relations}), 
			and exploit them to develop effective algorithms for approximating $h^+(G)$, $\widehat{h}^+(G)$
			and thus $h_{\max}(G)$,  all of which are combinatorially NP-hard to be analytically determined.

			\tikzstyle{startstop} = [rectangle, minimum width=1cm, minimum height=1cm,text centered, draw=black]
			\tikzstyle{io1} = [rectangle, trapezium left angle=80, trapezium right angle=100, minimum width=1cm, minimum height=1cm, text centered, draw=black]
			\tikzstyle{io2} = [trapezium,  rounded corners, trapezium left angle=100, trapezium right angle=100, minimum width=1cm, minimum height=1cm, text centered, draw=black]
			\tikzstyle{process} = [rectangle, minimum width=1cm, minimum height=1cm, text centered, draw=black]
			\tikzstyle{decision} = [circle, minimum width=1cm, minimum height=1cm, text centered, draw=black]
			\tikzstyle{decision2} = [ellipse, rounded corners=10mm, minimum width=2cm, minimum height=2cm, text centered, draw=black]
			\tikzstyle{arrow} = [thick,->,>=stealth]

			\begin{figure}[htbp]
				\centering
				\begin{tikzpicture}[scale=1, node distance=3.5cm]
					
					\node (maxcut) [startstop] { $h_{\max}$};
					
					\node (2Lap) [decision, below of=maxcut, xshift=0cm, yshift=1cm] {  
						$\Delta_2$
					};
					
					\node (mdual) [startstop][startstop, right of=maxcut, xshift=2cm, yshift=1cm] {  $\widehat{h}^+$
					};
					
					\node (dual) [startstop][startstop, below of=mdual, xshift=0cm, yshift=1cm] {  $h^+$
					};
					
					\node (Cheeger) [startstop][startstop, below of=dual, xshift=0cm, yshift=1cm] {  $h$
					};
					
					\node (hat1Lap) [startstop][startstop, right of=mdual, xshift=1cm, yshift=0cm] {  $\widehat{\Delta}^+_1$
					};
					
					\node (s1Lap) [startstop][startstop, right of=dual, xshift=1cm, yshift=0cm] {  $\Delta_1^+$
					};
					
					\node (1Lap) [startstop][startstop, right of=Cheeger, xshift=1cm, yshift=0cm] {  $\Delta_1$
					};
					
					\draw  [thick,->,>=stealth](maxcut) to node[above=5pt,left] {  Soto \cite{Soto2015} } node[anchor=north] {}
					(mdual);
					
					\draw  [thick,->,>=stealth](maxcut) to node[below=5pt,left] { Trevisan \cite{Trevisan2012}} 
					(dual);
					
					\draw  [thick,->,>=stealth](2Lap) to node[above=5pt,left] { Bauer-Jost \cite{BauerJost2013}} 
					(dual);
					
					\draw  [thick,->,>=stealth](2Lap) to node[below=5pt,left] {Chung \cite{Chung1997}} node[below=12pt] {Alon-Milman \cite{Alon1986}}
					(Cheeger);
					\draw  [thick,->,>=stealth](Cheeger) to node[anchor=south] { Chang \cite{Chang2015} } node[anchor=south,below] { Hein-B\"{u}hler \cite{HeinBuhler2010} } (1Lap);
					\draw  [thick,->,>=stealth](1Lap) to node[above=5pt,left] {{Theorem \ref{thm:sign-and-signless}} } (s1Lap);
					\draw  [thick,->,>=stealth](dual) to node[anchor=south] {Theorem \ref{th:1-eigen}
					} (s1Lap);
					\draw  [thick,->,>=stealth](mdual) to node[anchor=south] {Theorem \ref{th:1-eigen} } (hat1Lap);
					\draw  [thick,->,>=stealth](s1Lap) to node[above=5pt,left] {{Proposition \ref{pro:h+=1}}} node[below=10pt,left] {{\textcolor{black}{Theorem \ref{th:1-ieq}}}} (hat1Lap);
				\end{tikzpicture}
				\caption{Operators, constants and connections investigated in this paper. The explicit reference to the graph G is neglected here for simplicity.}
				\label{fig:relations}
			\end{figure}
			
			We first show that 
			the ``Rayleigh-type  quotients'' 
			emerging in Theorem~\ref{th:1-eigen} 
			are locally linear along any given direction,
			and then use such local linearity to develop 
			the inverse power algorithms for approximating $\mu_1^+$ and $\widehat{\mu}_1^+$.
			The resulting algorithms are denoted as ``{\rm{\bf{IP}}}" and ``$\widehat{\rm \bf{IP}}$", respectively. 
			It is proved that the iterative values of IP (resp. $\widehat{\text{IP}}$) are monotonically decreasing and converging to some eigenvalues of $\Delta_1^+$ (resp. $\widehat{\Delta}_1^+$),
			and the iterative points converge to the corresponding eigenvectors which are ternary cuts and directly serve as the approximate solutions for the (resp. modified) dual Cheeger problem.  
			We use {\rm{\bf{IP}}}  (resp. $\widehat{\rm \bf{IP}}$) to provide the approximate values of $h^+$ (resp. $\widehat{h}^+$) on G-set for the first time. After that, we replace the rounding procedure,  2-Thresholds Spectral
			Cut,  with {\rm{\bf{IP}}} or $\widehat{\rm \bf{IP}}$ in the RSC algorithm for the maxcut problem~\eqref{eq:maxcut0}
			and find that these enhanced RSC algorithms yield approximate solutions of higher quality than the original one on the standard testbed G-set. It should be pointed out that such efficient implementation of RSC for maxcut using  
			the equality relations revealed in Theorem~\ref{th:1-eigen}  has not been noticed
			previously to the best of our knowledge. \textcolor{black}{Meanwhile, we prove that 
				the lower bound of the worst-case approximation ratio of RSC for maxcut cannot be improved greater than $0.769$}.  
			
			The rest of this paper is organized as follows. 
			Section~\ref{sec:signless} presents the spectral theory for $\Delta_1^+$ and $\widehat{\Delta}_1^+$, along with their connections to $\Delta_1$ and $\Delta_2$. In Section~\ref{sec:local-IP}, we analyze the objective functions associated with the $\Delta_1^+$- and  $\widehat{\Delta}_1^+$-eigenproblems, and propose the inverse power algorithm for seeking the optimal solution of the (modified) dual Cheeger problem. Section~\ref{sec:maxcut} enhances the RSC algorithm for the maxcut problem using the equality relations revealed in Theorem~\ref{th:1-eigen}.  
			Section~\ref{sec:con} rounds the paper off with our conclusions.

			\bigskip

{\bf Acknowledgements.}
This work was supported by the National Key R~\&~D Program of China (No.~2022YFA1005102) and the National Natural Science Foundation of China (Grant Nos.~ 12401443, 12325112, 12288101). The authors would like to thank Professor Kung-Ching Chang 
for his long-term guidance, encouragement and support 
in mathematics, and useful comments on an earlier version of this paper.

			\section{Spectrum of signless $1$-Laplacians}
			\label{sec:signless}
			The nonlinear spectral theory of the 1-Laplacian $\Delta_1$ has been well documented \cite{Chang2015,ChangShaoZhang2015,ChangShaoZhang2017-am,HeinBuhler2010,SzlamBresson2010}.
			This section extends it to the signless $1$-Laplacians $\Delta_1^+$ and $\widehat{\Delta}_1^+$ in the same spirit of \cite{Chang2015,ChangShaoZhang2015}. For simplicity, we work on unweighted simple graphs, but all these results hold for graphs with positive weights.
			Without causing any ambiguity, the explicit reference to the graph $G$ is neglected sometimes hereafter such as $h_{\max}$, $h^+$, $\widehat{h}^+$.

			\begin{definition}[(Normalized) Ternary Vector]
				\label{def:ternary-cut}
				A vector $\vec x$ in $\mathbb{R}^n\setminus \{\vec 0\}$ is said to be a  {\sl ternary vector} if there exist two disjoint subsets $A$ and $B$ of $V$ such that $\vec x=\hat{\vec 1}_{A,B} \in X$ with
				\[ (\hat{\vec 1}_{A,B})_i=\begin{cases}
					1/\vol(A\cup B),&\text{ if }\;\; i\in A,\\
					-1/\vol(A\cup B),&\text{ if }\;\; i\in B,\\
					0,&\text{ if }\;\; i\not\in A\cup B.
				\end{cases}\]
				For any $\vec x\in\mathbb{R}^n\setminus \{\vec 0\}$, we use $\vec x_{+,-}$ to denote its \emph{normalized ternarization} 
				$\hat{\vec 1}_{D^+(\vec x),D^-(\vec x)}$ with $D^\pm(\vec x) = \{i\in V:\pm x_i>0\}$.
			\end{definition}
			
			\begin{lemma}\label{lem:eigen1}
				Let  $(\mu^+,\vec x)$ be an eigenpair of $\Delta^+_1$, and  let $\left(\widehat{\mu}^+, \vec{y}\right)$ be an eigenpair of $\widehat{\Delta}_1^+$.  Then we have:
				\begin{enumerate}
					\item $I^+(\vec x)=\mu^+\in [0,1]$;
					\item $(\mu^+,\vec x_{+,-})$ is also an eigenpair of $\Delta^+_1$;
					\item The distance between two distinct eigenvalues of $\Delta_1^+$ is at least $\frac{2}{n^2(n-1)^2}$;
					\item  $ I^+(\vec y)/\left(I(\vec y)+I^+(\vec y)\right)=\widehat{\mu}^+\in[0,1]$;
					\item $\left(\widehat{\mu}^+, \vec{y}_{+,-}\right)$ is also an eigenpair of $\widehat{\Delta}_1^+$;
					\item The distance between two distinct eigenvalues of $\widehat{\Delta}_1^+$ is at least $\frac{2}{n^2(n-1)^2}$.
				\end{enumerate}
			\end{lemma}
			
			\begin{proof}The first statement is standard, and thus we turn to the proof of the second statement. 
				Let $\hat{\vec x}=\vec x_{+,-}$. Then, by the definition of $\hat{\vec x}$, there holds $\sgn(\hat{x}_i)\supset \sgn(x_i)$, $i=1,2,\cdots,n$. Now, we begin to verify that $\sgn(\hat{x}_i+\hat{x}_j)\supset \sgn(x_i+x_j)$, whenever $\{i,j\}\in E$. If $x_ix_j\ge 0$, then $\hat{x}_i\hat{x}_j\ge0$ and thus it is easy to check that $\sgn(\hat{x}_i+\hat{x}_j)=\sgn(x_i+x_j)$. If $x_ix_j<0$, then $\hat{x}_i\hat{x}_j<0$ and hence $\sgn(\hat{x}_i+\hat{x}_j)=\sgn(0)=[-1,1]\supset \sgn(x_i+x_j)$. Therefore, we have proved that $\sgn(\hat{x}_i+\hat{x}_j)\supset \sgn(x_i+x_j)$, for any $\{i,j\}\in E$. 
				Since $(\mu^+,\vec x)$ is an eigenpair, by the componentwise formulation of the eigenequation \eqref{eq:L1-plus}, there exists $z_{ij}(\vec x)\in \sgn(x_i+x_j)$ satisfying
				\[
				z_{ji}(\vec x)=z_{ij}(\vec x),\,\forall\,\textcolor{black}{\{j,i\}\in E}
				\;\;
				\text{and}
				\;\;
				\sum_{\textcolor{black}{j:\{j,i\}\in E}} z_{ij}(\vec x)\in \mu^+ d_i \sgn (x_i),\,\, i=1, \cdots, n.
				\]
				If one takes $z_{ij}(\hat{\vec x})=z_{ij}(\vec x), \forall\,\textcolor{black}{\{j,i\}\in E}$, then
				$z_{ij}(\hat{\vec x})\in \sgn(\hat{x}_i+\hat{x}_j)$ and they also satisfy
				\[
				z_{ji}(\hat{\vec x})=z_{ij}(\hat{\vec x}),\,\forall\,\textcolor{black}{\{j,i\}\in E}
				\;\;
				\text{and}
				\;\;
				\sum_{\textcolor{black}{j:\{j,i\}\in E}} z_{ij}(\hat{\vec x})\in \mu^+ d_i \sgn (\hat{x}_i),\,\, i=1, \cdots, n.
				\]
				Consequently, $(\mu^+,\hat{\vec x})$ is an eigenpair, too. This completes the proof of the second statement. 
				
				Up to now, we have known that,  given two different $\Delta_1^+$-eigenvalues $\mu^+$ and $\tilde{\mu}^+$, there exist
				$ A_1,A_2, B_1,B_2\subset V$ such that $\mu^+=1-\frac{2|E(A_1,A_2)|}{\vol (A_1\cup A_2)}$ and $\tilde{\mu}^+=1-\frac{2|E(B_1,B_2)|}{\vol (B_1\cup B_2)}$
				through the normalized ternarization of corresponding $\Delta_1^+$-eigenvectors. Accordingly, we obtain
				\begin{align*}
					|\mu^+-\tilde{\mu}^+| &=
					\left|\frac{2|E(A_1,A_2)|}{\vol (A_1\cup A_2)}-\frac{2|E(B_1,B_2)|}{\vol (B_1\cup B_2)}\right|
					\\&=2\frac{\left||E(A_1,A_2)|\vol (B_1\cup B_2)-|E(B_1,B_2)|\vol (A_1\cup A_2)\right|}{\vol (A_1\cup A_2)\vol (B_1\cup B_2)}
					\\&\ge \frac{2}{\vol (A_1\cup A_2)\vol (B_1\cup B_2)}
					\ge \frac{2}{n^2(n-1)^2},
				\end{align*}
				which proves the third statement. 
				
				The proof of the last three statements related to the modified signless 1-Laplacian $\widehat{\Delta}_1^+$ is similar to the first three ones, and thus we will not write down the details.
			\end{proof}
			
			
			\begin{proof}[Proof of  Theorem \ref{th:1-eigen}] We shall use Lemma \ref{lem:eigen1} to prove {\color{black}Equations~}\eqref{eq:min} and  \eqref{eq:modified-min}. 
				We prove the second equality of {\color{black}Equation~}\eqref{eq:min} at the first place. 
				Since both  $I^+(\vec x)$ and $\|\vec x\|$  are positively $1$-homogenous, we have 
				$$
				\frac{I^+(\vec  x)}{\|\vec  x\|}= I^+(\vec  x)|_X, \quad \forall\, \vec  x\in X, 
				$$
				implying that the minimum of the function $I^+(\vec x)$ on $X$ is  a critical value and thus an eigenvalue of $\Delta^+_1$, 
				and then is equal to $\mu^+_1$.
				
				Then we turn out to prove the first equality of {\color{black}Equation~}\eqref{eq:min}. It is easy to calculate that
				\begin{align*}
					1-\frac{2|E(V_1,V_2)|}{\vol(V_1\cup V_2)}&=\frac{2|E(V_1,V_1)|+2|E(V_2,V_2)|+|E(V_1\cup V_2,(V_1\cup V_2)^c)|}{\vol(V_1\cup V_2)}
					\\&=I^+(\hat{\vec 1}_{V_1,V_2}),
				\end{align*}
				where $\hat{\vec 1}_{V_1,V_2}$ is a ternary vector in $X$ \textcolor{black}{as defined in Definition \ref{def:ternary-cut}}.
				Therefore, we have
				\begin{align*}
					1-h^+(G)&=1-\max\limits_{V_1\cap V_2=\emptyset,V_1\cup V_2\ne \emptyset} \frac{2|E(V_1,V_2)|}{\vol(V_1\cup V_2)}
					\\&=\min\limits_{V_1\cap V_2=\emptyset,V_1\cup V_2\ne \emptyset}\left(1-\frac{2|E(V_1,V_2)|}{\vol(V_1\cup V_2)}\right)
					\\&=\min\limits_{V_1\cap V_2=\emptyset,V_1\cup V_2\ne \emptyset}I^+(\hat{\vec 1}_{V_1,V_2})
					\ge \min\limits_{\vec x\in X}I^+(\vec x).
				\end{align*}

				On the other hand, let $\vec x^0$ be a minimizer of $I^+(\vec x)$ on $X$. 
				Then $\vec x^0$ must be a critical point of $I^+|_X$ and thus a $\Delta^+_1$-eigenvector with respect to $\mu^+_1$. Accordingly, by Lemma \ref{lem:eigen1} we deduce that there is a ternary vector $\hat{\vec x}^0$ which is also a $\Delta^+_1$-eigenvector corresponding to  $\mu^+_1$.
				By the construction of $\hat{\vec x}^0$, there exist $V_1^0$ and $V_2^0$ such that
				$$
				I^+(\hat{\vec x}^0)=1-\frac{2|E(V_1^0,V_2^0)|}{\vol(V_1^0\cup V_2^0)},
				$$
				which immediately yields
				\begin{align*}
					\min\limits_{\vec x\in X}I^+(\vec x)&=I^+(\hat{\vec x}^0)
					\\&\ge\min\limits_{V_1\cap V_2=\emptyset,V_1\cup V_2\ne \emptyset}\left(1-\frac{2|E(V_1,V_2)|}{\vol(V_1\cup V_2)}\right)
					\\&=1-h^+(G).
				\end{align*}
				The proof of  \eqref{eq:min} is then completed.
				
				
				
				Let $\mu^*=\min\limits_{\vec  x\in\mathbb{R}^n\backslash\{\vec 0\}}\frac{I^+(\vec x)}{I(\vec x)+I^+(\vec x)}\in [0,1]$,
				$\vec x^*\in\argmin\limits_{\vec  x\in\mathbb{R}^n\backslash\{\vec 0\}}\frac{I^+(\vec x)}{I(\vec x)+I^+(\vec x)}$, 
				and $\langle\cdot,\cdot\rangle$ denote the standard inner product in $\mathbb{R}^n$. We first prove the second equality of  {\color{black}Equation~}\eqref{eq:modified-min}. Fixing any $\vec s\in\partial I(\vec x^*)$, let   
				$$
				f(\vec x)=(1-\mu^*)I^+(\vec x)-\mu^* \langle\vec x,\vec s\rangle.
				$$
				Since the function $I(\cdot)$ is convex and homogeneous of degree one, we have 
				$$
				I(\vec{y}) =\langle\vec{y}, \vec{s}\rangle, \quad I(\vec{x}) \geq \langle\vec{x}, \vec s\rangle, 
				\quad \forall\, \vec{s} \in \partial I(\vec{y}), \quad \forall\, \vec{x}, \vec{y} \in \mathbb{R}^n.
				$$
				Then, $\forall\, \vec{x}\in \mathbb{R}^n$, $f(\vec x)\ge (1-\mu^*)I^+(\vec x)-\mu^*I(\vec x)\ge 0=f(\vec x^*)$. This proves that $\vec x^*\in\argmin_{\vec x\in\mathbb{R}^n}\{f(\vec x)\}$. 
				Thus, $\vec 0\in\partial f(\vec x^*)$ is satisfied. 
				That is, there exists $\vec z\in\partial I^+(\vec x^*)$ 
				such that $\vec 0=(1-\mu^*)\vec z-\mu^*\vec s$, which indicates that $(\mu^*,\vec x^*)$ is also an eigenpair of $\widehat{\Delta}_1^+$ and $\mu^*\ge \widehat{\mu}_1^+$.  On the other hand, we have $\widehat{\mu}_1^+ \ge \mu^*$ directly from Lemma \ref{lem:eigen1} and thus arrives at the second equality of {\color{black}Equation~}\eqref{eq:modified-min}. 
				
				We turn to prove the first equality of {\color{black}Equation~}\eqref{eq:modified-min}. Assuming that $(V_1,V_2)$ is a solution to the modified dual Cheeger constant on $G$, then the ternary vector $\hat{\vec 1}_{V_1,V_2}=(\vec1_{V_1}-\vec1_{V_2})/\vol(V_1\cup V_2)$ satisfies
				\[
				1-\widehat{h}^+(G)=\frac{I^+(\hat{\vec 1}_{V_1,V_2})}{I(\hat{\vec 1}_{V_1,V_2})+I^+(\hat{\vec 1}_{V_1,V_2})} \geq \widehat{\mu}_1^+.
				\]
				On the other hand, Lemma \ref{lem:eigen1} ensures that there exists a ternary eigenvector corresponding to $\widehat{\mu}_1^+$, which yields $1-\widehat{h}^+(G)\leq \widehat{\mu}_1^+$. The proof of   \eqref{eq:modified-min} is then completed.
			\end{proof}
			
			\begin{remark}
				The second equalities of both {\color{black}Equation~}\eqref{eq:min} and {\color{black}Equation~}\eqref{eq:modified-min} can be directly proved by the set-pair Lov{\'a}sz extension \cite{CSZZ-Lovasz18}, while the above proof presented in this paper is based on the spectral point of view (i.e., Lemma \ref{lem:eigen1}).
			\end{remark}
			
			It is known that a
			connected graph $G$ is bipartite if and only if $h^+(G)=1$ \cite{BauerJost2013}, and then by Theorem \ref{th:1-eigen}, $h^+(G)=1$ if and only if $\mu^+_1=0$. This derives
			that a connected graph $G$ is bipartite if and only if $\mu^+_1=0$. 
			Moreover, we have 
			\begin{proposition}
				\label{pro:h+=1}
				For a 
				graph $G$, we have
				$h^+(G)=1$ $\Longleftrightarrow$ $\widehat{h}^+(G)=1$ $\Longleftrightarrow$  $\mu^+_1=0$  $\Longleftrightarrow$  $\widehat{\mu}_1^+=0$  $\Longleftrightarrow$   $\lambda_n = 2$ $\Longleftrightarrow$ $G$ has a bipartite connected component. 
				If $G$ is further assumed to be connected, then 
				\begin{enumerate}
					\item  $h_{\max}(G)=1\Longleftrightarrow h^+(G)=1\Longleftrightarrow\widehat{h}^+(G)=1 \Longleftrightarrow \lambda_n=2 \Longleftrightarrow$  $G$ is bipartite;
					\item  $h_{\max}(G)=\lambda_n/2$ $\Longleftrightarrow$ $h^+(G)=\lambda_n/2$ $\Longleftrightarrow$ $\widehat{h}^+(G)=\lambda_n/2$.
				\end{enumerate}
			\end{proposition}
			
			\begin{proof}
				We only prove the last statement while the remaining \textcolor{black}{claims are} direct consequences of the combination of \textcolor{black}{Theorems \ref{th:1-ieq} and \ref{th:1-eigen}}. 
				
				Suppose that the set pair $(V_1,V_2)$ solves the modified dual Cheeger problem \eqref{eq:modified-dual-Cheeger-constant},
				and the indicative vectors are $\vec 1_{V_1}$ and $\vec 1_{V_2}$, respectively. A straightforward calculation gives
				\begin{align}
					\frac{4|E(V_1,V_2)|+|E\left(V_1\cup V_2,(V_1\cup V_2)^c\right)|}{\vol(V_1\cup V_2)} &= \frac{I_2(\vec 1_{V_1}-\vec 1_{V_2})}{\|\vec 1_{V_1}-\vec 1_{V_2}\|_{2,d}}, \label{eq:lambda_n}\\ 
					\frac{4|E(V_1,V_2)|+2(1-\widehat{h}^+)|E\left(V_1\cup V_2,(V_1\cup V_2)^c\right)|}{\vol(V_1\cup V_2)} &= 2\widehat{h}^+. \label{eq:2hhat+}
				\end{align}
				Using the Courant-Fischer theorem and the fact that $1/2\le\widehat{h}^+\le1$, we are able to arrive at $\widehat{h}^+\le {\lambda_n}/{2}$
				after comparing {\color{black}Equation~}\eqref{eq:lambda_n} with {\color{black}Equation~}\eqref{eq:2hhat+}, and thus 
				\begin{equation}
					\label{leq:lambda_n1}
					h_{\max}\le h^+\le \widehat{h}^+\le \frac{\lambda_n}{2}.
				\end{equation}
				It follows directly from Eq. \eqref{leq:lambda_n1} that  $h_{\max}=\lambda_n/2$ $\Longrightarrow$ $h^+=\lambda_n/2$ $\Longrightarrow$ $\widehat{h}^+=\lambda_n/2$. Next we prove the opposite direction. 
				Once $\widehat{h}^+=\lambda_n/2$, $|E\left(V_1\cup V_2,(V_1\cup V_2)^c\right)|=0$ is naturally satisfied because $2-2\widehat{h}^+=2-\lambda_n\leq (n-2)/(n-1)<1$. Since $G$ is connected, we have $V_1\cup V_2=V$, and thus 
				$(V_1,V_2)$ is also a solution to the maxcut problem \eqref{eq:maxcut0}, thereby implying that $h_{\max}=h^+=\lambda_n/2$.
			\end{proof}

			\begin{proof}[Proof of Theorem \ref{th:1-ieq}]  
				We only need to prove 
				\[
				\widehat{h}^+(G)\le \frac{\lambda_n-h^+(G)}{1+\lambda_n-2h^+(G)} \le \frac{\lambda_n-\widehat{h}^+(G)}{1+\lambda_n-2\widehat{h}^+(G)}\le \frac{\lambda_n}{2}.
				\]
				The second and last inequalities are respectively equivalent to $\widehat{h}^+\geq h^+$ and $\widehat{h}^+\le {\lambda_n}/{2}$, 
				both of which have been verified in the proof of Proposition \ref{pro:h+=1}. It remains to verify the first inequality. 
				Let $(V_1,V_2)$ solve the modified dual Cheeger problem \eqref{eq:modified-dual-Cheeger-constant}.
				Then, combining {\color{black}Equations~}\eqref{eq:lambda_n} and \eqref{eq:2hhat+} yields
				\begin{equation}
					\label{eq:2hhat+1}
					\widehat{h}^+=\left(2\widehat{h}^+-1\right)\frac{2|E(V_1,V_2)|}{\vol(V_1\cup V_2)}+\left(1-\widehat{h}^+\right) \frac{I_2(\vec 1_{V_1}-\vec 1_{V_2})}{\|\vec 1_{V_1}-\vec 1_{V_2}\|_{2,d}}.
				\end{equation}
				The first inequality emerges accordingly 
				after applying the Courant-Fischer theorem and {\color{black}Equation~}\eqref{eq:dual-Cheeger-constant} into {\color{black}Equation~}\eqref{eq:2hhat+1}. 
			\end{proof}

			\begin{remark}
				\label{remark:maxcut}
				The dual Cheeger constants $h^+(G)$ and $\widehat{h}^+(G)$ serve as new sharper upper bounds for $h_{\max}(G)$ from {\color{black}Equation~}\eqref{eq:relation-inequality}.  All known upper bounds relate to either the maximal eigenvalue $\lambda_n$ of the normalized Laplacian or a finer constant
				\begin{equation}\label{eq:phi}
					\varphi:=\inf\limits_{\sum_{i=1}^n u_i=0}\text{maximum eigenvalue of }D^{-1/2}(\Delta_2+\diag(u_1,\cdots,u_n)) D^{-1/2},
				\end{equation}
				both of which are often used in the linear spectral graph theory,
				through $h_{\max}\leq{\lambda_n}/{2}$ and
				$h_{\max}\leq{\varphi}/{2}$ \cite{DelormePoljak1993,PoljakRendl1995}.
				Actually, as shown in Table~\ref{table:phih}, it is apparent that the upper bounds $h^+$ and $\widehat{h}^+$ are not explicitly correlated with $\varphi/2$, while all three upper bounds are sharper than the classical bound $\lambda_n/2$.
				Note in passing that a lower bound for $h_{\max}$
				was given in terms of the chromatic number \cite{PoljakTuza1995}.
			\end{remark}

			\begin{table}
				\renewcommand\arraystretch{1.4}
				\caption{Three graphs and corresponding values for
					\textcolor{black}{$h_{\max}$, $h^+$, $\widehat{h}^+$, $\lambda_n$},
					and $\varphi$ given in {\color{black}Equation~}\eqref{eq:phi}.
					It can be verified that the minimum order of graph \textcolor{black}{without isolated vertices} satisfying \textcolor{black}{$h^+=\widehat{h}^+=\varphi/2$, $h^+=\widehat{h}^+<\varphi/2$, and $h^+=\widehat{h}^+>\varphi/2$} is 2, 3, and 5, respectively, as shown in the first row.}
				\label{table:phih}
				\centering
				\begin{tabular}{c|c|c|c}
					\hline\hline
					{graph}
					& \;\;\;\;\;\;\;\;\;\;
					\begin{tikzpicture}
						\node (A) at (0,-1) {};
						\node (B) at (0,1)  {};
						\draw (A) to (B);
						\draw (A) circle(0.16);
						\draw (B) circle(0.16);
					\end{tikzpicture} \;\;\;\;\;\;\;\;\;\;
					& \;\;\;\;\begin{tikzpicture}[auto]
						\node (A) at (0,0) {};
						\node (B) at (2,1) {};
						\node (C) at (2,-1) {};
						\draw (A) to (B);
						\draw (C) to (B);
						\draw (C) to (A);
						\draw (A) circle(0.16);
						\draw (B) circle(0.16);
						\draw (C) circle(0.16);
					\end{tikzpicture} \;\;\;\; &
					\begin{tikzpicture}[auto]
						\node (A) at (0,0) {};
						\node (B) at (2,1) {};
						\node (C) at (2,-1) {};
						\node (D) at (3,-1) {};
						\node (E) at (3,1) {};
						\draw (A) to (B);
						\draw (C) to (B);
						\draw (C) to (A);
						\draw (D) to (E);
						\draw (A) circle(0.16);
						\draw (B) circle(0.16);
						\draw (C) circle(0.16);
						\draw (D) circle(0.16);
						\draw (E) circle(0.16);
					\end{tikzpicture}
					\\
					\hline
					$h_{\max}$ & $1$ & $2/3$  & $3/4$ \\
					\hline
					$h^+$ & $1$ & $2/3$ & $1$ \\
					\hline
					\textcolor{black}{$\widehat{h}^+$}&\textcolor{black}{$1$}&\textcolor{black}{$2/3$}&\textcolor{black}{$1$}\\
					\hline
					$\varphi/2$ & $1$& $3/4$ &  $<0.9$ \\
					\hline
					$\lambda_n/2$ & $1$ & $3/4$ & $1$ \\
					\hline\hline
				\end{tabular}
			\end{table}

			Next{\color{black},} we study the construction of the set of eigenvectors associate to 
			a given $\Delta^+_1$-eigenvalue. For any $\vec x\in X$, we introduce the simplex 
			$$\triangle(\vec x)=\{\vec y\in X:\sgn(y_i)=\sgn(x_i),~i=1,2,\cdots,n\}.$$
			Then $X=\bigcup_{\vec x\in X}\triangle(\vec x)$ can be viewed as a simplicial complex. 
			Note that  the family of hyperplanes $$\{\pi_{\{i,j\}}|\textcolor{black}{\{i,j\}\in E}\},\;\;\text{where }\pi_{\{i,j\}}=\{\vec x\in \mathbb{R}^n| x_i+x_j=0\},$$
			divides each simplex $\Delta(\vec x)$ \textcolor{black}{for $\forall\, \vec x\in X$} into more refined simplices of the form
			$$
			\triangle^+(\vec x)=\{\vec y\in \triangle(\vec x): \sgn(y_i+y_j)=\sgn(x_i+x_j),\,~\textcolor{black}{\{j,i\}\in E},\,~i=1,2,\cdots,n\}.
			$$
			Therefore, $X^+:=\bigcup_{\vec x\in X}\triangle^+(\vec x)$ is a refined simplicial complex of $X$, which coincides with $X$ in the sense of set. 
			As an example, {\color{black}Figure~}\ref{fig:XJP2:1} cartoons
			the complex $X$ and the refined complex $X^+$ for
			the path graph with two vertices. 
			
			\begin{figure}[htbp]
				\centering
				
				\begin{tikzpicture}[
					scale=1,
					IS/.style={blue, thick},
					LM/.style={red, thick},
					axis/.style={very thick, ->, >=stealth', line join=miter},
					important line/.style={thick}, dashed line/.style={dashed, thin},
					every node/.style={color=black},
					dot/.style={circle,fill=black,inner sep=2pt,
						outer sep=4pt},
					]
					
					\draw[->] (-3,0) -- (3,0) node[anchor=north] {$x_1$};
					
					\draw (2.3,0.3) node {$(1,0)$};
					\draw (0.5,2.1) node {$(0,1)$};
					\draw (-1.1,2.2) node {$x_1+x_2=0$};

					\node[dot] at (0,2) (int1) {};
					\node[dot] at (0,-2) (int1) {};
					\node[dot] at (2,0) (int1) {};
					\node[dot] at (-2,0) (int1) {};
					\node[dot] at (-1,1) (int1) {};
					\node[dot] at (1,-1) (int1) {};

					
					\draw[->] (0,-3) -- (0,3) node[anchor=east] {$x_2$};

					\draw (-2.5,2.5) -- (2.5,-2.5);
					\draw[very thick,dotted](2,0)--(0,2);
					\draw[very thick,dotted](-2,0)--(0,-2);
					\draw[very thick,dashed] (-2,0) -- (0,2);
					\draw[very thick,dashed] (2,0) -- (0,-2);
					
					
				\end{tikzpicture}
				
				\caption{\small The complex $X$ of the path graph with two vertices consists of four 0-cells (the four vertices of the square) and four 1-cells (the four sides of the square),
					while the corresponding refined complex $X^+$ consists of six 0-cells (black dots) and four small 1-cells (dashed lines) and two big 1-cells (dotted lines).}
				\label{fig:XJP2:1}
			\end{figure}

			By Lemma \ref{lem:eigen1}, we have:
			
			\begin{proposition}\label{lem:eigen2}
				If $(\mu^+,\vec x)$ is an eigenpair of $\Delta_1^+$, then $(\mu^+,\vec y)$ is also an eigenpair for any $\vec y\in \overline{\triangle^+(\vec x)}$, where $\overline{\triangle^+(\vec x)}$ is the closure of the set $\triangle^+(\vec x)$.
			\end{proposition}
			
			At the end of this section, we establish  a relation between the spectra of 1-Laplacian $\Delta_1$ and its signless version $\Delta^+_1$.
			Lemma~\ref{lem:eigen1} indeed implies that the number of $\Delta^+_1$-eigenvalues is finite.  
			Precisely, there are at most $\frac{n^2(n-1)^2}{2}+1$ different eigenvalues,  and
			then all the eigenvalues can be ordered as: $0\le\mu^+_1<
			\mu^+_2<\cdots\le 1$.
			However, due to the high nonlinearity,  
			we do not know a priori how many eigenvalues exist. Nevertheless it is always possible to select $n$ of them as representatives 
			of the whole spectrum. 
			In fact, Chang \cite{Chang2015} used the  Liusternik-Schnirelmann theory and the concept of Krasnoselskii genus to select the sequence of min-max eigenvalues of $\Delta_1$.  Parallel to this, we shall apply the Liusternik-Schnirelmann theory to $\Delta^+_1 $. Now, $I^+(\vec x)$ is even, and $X$ is symmetric. Let $T\subset X$ be a nonempty  symmetric subset, \ie $-T=T$. The integer valued function $\gamma: T \mapsto \mathbb{Z}^+$, which is called the Krasnoselski genus of $T$ \cite{Chang1985,Rabinowitz1986}, is defined to be:
			\begin{equation*}
				\gamma(T) =\min\limits\{k\in\mathbb{Z}^+: \exists\; \text{odd continuous}\; h: T\to \mathbb{S}^{k-1}\}. 
			\end{equation*}
			Obviously, the genus is a topological invariant. Let us define
			\begin{equation}\label{eq:c_k}
				c_k^+=\inf_{\gamma(T)\ge k} \max_{\vec x\in T} I^+(\vec x),\quad k=1,2, \cdots n.
			\end{equation}
			Along the same line of \cite{Chang2015,CSZZ18prepare}, it can be proved that these $c_k^+$ are critical values of $I^+(\vec x)$,   and thus they are eigenvalues of $\Delta_1^+$.  In addition, one has
			$$c_1^+\le c_2^+ \le \cdots \le c_n^+,$$
			and if $0\le\cdots\le c_{k-1}^+<c_k^+=\cdots=c_{k+r-1}^+<c_{k+r}^+\le \cdots\le 1$,  the (variational)  multiplicity of $c_k^+$ is simply defined to be $r$.  
			It is worth noting that for a general $\Delta_1^+$-eigenvalue $\mu^+$, 
			we use $\gamma(K_{\mu^+})$ to define the multiplicity of $\mu^+$, where $K_{\mu^+}=\{\vec x\in X:(\mu^+,\vec x)\text{ is an eigenpair of }\Delta_1^+\}$. However, in order to distinguish the two types of multiplicities,  
			we usually use the name $\gamma$-multiplicity instead of multiplicity for a general eigenvalue. In consequence,  
			we are in a position to show 
			a relation between $\Delta_1$ and $\Delta_1^+$. 		
			
			\begin{theorem}\label{thm:sign-and-signless}
				The following three statements are equivalent:
				\begin{itemize}
					\item[(a)] $G$ is bipartite;
					\item[(b)]  the spectra of $\Delta_1^+$ and $\Delta_1$ coincide (counting $\gamma$-multiplicity);
					\item[(c)]  the min-max eigenvalues of $\Delta_1^+$ agree with that of $\Delta_1$, i.e.,  $\forall\,k\in\{1,\cdots,n\}$, $c_k^+=c_k$. 
					Here $c_k$ is defined in the same way by using  $I(\vec x)$ instead  of $I^+(\vec x)$ in  {\color{black}Equation~}\eqref{eq:c_k}.
				\end{itemize}
			\end{theorem}
			
			\begin{proof}
				It is  known that the multiplicity of the eigenvalue 0 of $\Delta_1$ equals the number of connected  components of $G$. 
				Below we show a ``dual'' claim  of this result.
				\begin{enumerate}
					\item[Claim.] The multiplicity of the eigenvalue 0 of $\Delta_1^+$ equals the number of bipartite components of $G$. 
					
					Proof: 
					It is easy to see that the eigenvectors corresponding to  the eigenvalue 0 coincide with the zeros of $I^+(\cdot)$, i.e., $\{\vec x\in\mathbb{R}^n: x_i+x_j=0,\forall \{i,j\}\in E\}$. For simplicity, we denote $S_\pm=\{\vec x\in\mathbb{R}^n: x_i\pm x_j=0,\forall \{i,j\}\in E\}$.  
					Note that both $S_+$ and $S_-$ are linear subspaces. 
					If the graph $G$ has no bipartite component, then the equation system $x_i+x_j=0,\forall \{i,j\}\in E$ has only zero solution $x_i=0,\forall i$. Let $(V_1,E_1),\cdots,(V_k,E_k)$ be the bipartite components of $G$. Then each bipartite component $(V_i,E_i)$ provides a nonzero solution $\vec x=\vec 1_{V_i^+}-\vec 1_{V_i^-}$, where $V_i^+$ and $V_i^-$ are the two parts of $V_i$. Therefore, $S_+=\{\sum_{i=1}^kt_i(\vec 1_{V_i^+}-\vec 1_{V_i^-}):t_i\in\mathbb{R},i=1,\cdots,k\}$ and $\dim S_+=k$, 
					which completes the proof.
				\end{enumerate}
				
				We are able to prove Theorem \ref{thm:sign-and-signless}. 
				If the spectra of $\Delta_1^+$ and $\Delta_1$ coincide, or if $c_i^+=c_i$ for any  $i=1,\cdots,n$, then 
				the multiplicity of the eigenvalue 0 of $\Delta_1^+$ agrees with that of $\Delta_1$. 
				Then, by  the above claim, 
				the number of connected bipartite components of $G$ is equal to the number of connected  components of $G$. 
				This means that each  connected component of $G$ is bipartite! Hence, $G$ is bipartite too. Accordingly,  (b) $\Rightarrow$ (a), and (c) $\Rightarrow$ (a).
				
				Next, we shall prove   (a) $\Rightarrow$ (b) \& (c).  Suppose that $G=(V,E)$ is a bipartite graph with $V_+$ and $V_-$ being the two parts.  Taking $$M=\diag(a_1,\ldots,a_n),\;\; \text{ where }\; a_i=\begin{cases}
					1,&i\in V_+,\\ -1,& i\in V_-,
				\end{cases}$$
				it can be verified that $(\lambda,\vec x)$ is a $\Delta_1$-eigenpair if and only if $(\lambda,M \vec x)$ is a $\Delta_1^+$-eigenpair. 
				In consequence, considering $M$ as a linear isomorphism, we have  $M(K_\lambda(\Delta_1))=K_\lambda(\Delta_1^+)$ and $M(K_\lambda(\Delta_1^+))=K_\lambda(\Delta_1)$, for any eigenvalue $\lambda$ of $\Delta_1$ or $\Delta_1^+$, where $K_\lambda(\Delta_1)$ is the set of all $\Delta_1$-eigenvectors corresponding to the eigenvalue $\lambda$. 
				Denote by $m(\lambda,\Delta_1^+)$ (resp., $m(\lambda,\Delta_1)$) the multiplicity of the eigenvalue $\lambda$ of $\Delta_1^+$ (resp., $\Delta_1$). Then 
				$$m(\lambda,\Delta_1)=\gamma(K_\lambda(\Delta_1))=\gamma\left( M(K_\lambda(\Delta_1))\right)=\gamma(K_\lambda(\Delta_1^+))=m(\lambda,\Delta_1^+),$$
				where we used the fact that the genus of a set is invariant under any linear isomorphism. 
				In addition, for any  $i=1,\cdots,n$, \begin{align*}
					c_i^+&=\inf_{\gamma(S)\ge i} \max_{\vec x\in S} I^+(\vec x)=\inf_{\gamma(S)\ge i} \max_{\vec x\in M(S)} I^+(\vec x)=\inf_{\gamma(S)\ge i} \max_{\vec y\in S} I^+(M\vec y)
					\\&=\inf_{\gamma(S)\ge i} \max_{\vec y\in S} I(\vec y)=c_i,\end{align*}
				which means that the variational eigenvalues of $\Delta_1$ and $\Delta_1^+$ coincide, and furthermore, their  variational multiplicities also coincide. The proof is then completed.
			\end{proof}
			
			Theorem \ref{thm:sign-and-signless} provides a pretty nontrivial variant for the well-known relation between Laplacian and signless Laplacian on graphs. 
			
				%

				\section{Local analysis and the inverse power method}
				\label{sec:local-IP}
				
				From previous sections, we find that $I^+(\vec x)$ plays a central role in the study of the spectral theory for both $\Delta^+_1$ and $\widehat{\Delta}_1^+$. Before going to the numerical \textcolor{black}{part}, we shall analyze some local properties of $I^+(\vec x)$, by which the inverse power algorithm,  originally designed for the linear eigenvalue problem and recently extended to the Cheeger cut problem \cite{HeinSetzer2011}, can be proposed.
				
				The following elementary fact is well known.
				Let $g: [-1, 1]\to \mathbb{R}$ be a convex function, then $\partial g(0)=[g'(0,-), g'(0,+)],\, \mbox{or}\,\, [g'(0,+), g'(0,-)]$, where
				\begin{equation}
					g'(0,\pm)=\lim_{t\to \pm 0}\frac{g(t)-g(0)}{t}.
				\end{equation}
				
				\begin{lemma}\label{lemma:1-dim-PL-gradient}
					Let $g:[-1, 1]\to \mathbb{R}$ be a convex piecewise linear function. Then
					$$
					g(t)=g(0)+t
					\max_{s\in \partial g(0)}s
					$$
					holds for sufficiently small $t>0$.
				\end{lemma}
				\begin{proof}
					On one hand, by the definition of the sub-differential, we have
					$$ g(t)\ge g(0)+rt,\,\,\,\forall\, r\in \partial g(0),$$
					which implies that
					$$ g(t)\ge g(0)+t \max_{s\in \partial g(0)}s.$$
					On the other hand, if $g$ is piecewise linear, then $\exists\, r\in \mathbb{R}$ such that $g(t)= g(0)+rt$ for small $t>0$. So, we have
					$$ g'(0, +)=r,$$
					and then
					$r= \max _{s\in \partial g(0)}s$. The proof is completed. 
				\end{proof}
				
				\begin{theorem}
					Let $f: \mathbb{R}^n\to \mathbb{R}$ be a convex piecewise linear function. Then $\forall\, \vec v\in \mathbb{R}^n\setminus\{\vec 0\}$, for sufficiently small $t>0$,  we have
					$$ f(\vec a+t\vec v)=f(\vec a)+t   \max_{\vec p\in \partial f(\vec a)}\textcolor{black}{\langle\vec v, \vec p\rangle}.$$
				\end{theorem}
				\begin{proof}
					We only need to define $g_{\vec v}(t)=f(\vec a+t\vec v),\,\,\forall\, \vec v\in \mathbb{R}^n\setminus\{\vec 0\}$, and apply Lemma \ref{lemma:1-dim-PL-gradient} to $g_{\vec v}$.
				\end{proof}

				Applying the above theorem to the functions $I^+(\vec x)$ and $\|\vec x\|$, we obtain:

				\begin{corollary}
					Given $\vec a,\vec v\in \mathbb{R}^n\setminus \{\vec 0\}$, for sufficiently small $t>0$, we have
					\begin{align*}
						I^+(\vec  a+t\vec v)-I^+(\vec   a) &= t\cdot\max\limits_{\vec p\in\partial I^+(\vec   a)}\textcolor{black}{\langle\vec  p,\vec v\rangle}, \\
						\|\vec a+t\vec v\|-\|\vec a\| &=t\cdot\max\limits_{\vec p\in\partial \|\vec a\|}\textcolor{black}{\langle\vec p,\vec v\rangle}.
					\end{align*}
				\end{corollary}

				\begin{theorem}\label{th:pre_IP}
					Given $\vec a\in X$ and $\vec q\in\partial \|\vec a\|$,  there holds: (1) $\Rightarrow$ (2) $\Rightarrow$ (3), where the statements (1), (2) and (3) are claimed as follows.
					
					(1) $\vec a$ is not an eigenvector of $\Delta^+_1$.
					
					(2) $\min\limits_{\textcolor{black}{\vec x}\in \overline{B(\vec a,\delta)}}\left( I^+(\vec  x)-\lambda \textcolor{black}{\langle\vec q,\vec x\rangle}\right)<0$,
					where $\lambda=I^+(\vec  a)$, and $\delta$ can be taken as $$\min\left( \{ |a_{i}+a_{j}|,|a_i|:i,j=1,2,\cdots,n\}\backslash \{0\}\right).$$
					
					(3) $I^+(\vec x_0/\|\vec x_0\|)<I^+(\vec a)$, where
					\begin{equation}\label{eq:minsolve}
						\vec x_0\in \argmin \limits_{\textcolor{black}{\vec x}\in \overline{B(\vec a,\delta)}}\left( I^+(\vec  x)-\lambda \textcolor{black}{\langle\vec q,\vec x\rangle}\right).
					\end{equation}
				\end{theorem}
				\begin{proof}
					(1) $\Rightarrow$ (2):
					
					Let $\vec q\in\partial\|\vec a\|$ be fixed. We suppose the contrary, that $\vec a$ is not an eigenvector of $\Delta^+_1$, and $\min\limits_{\vec x\in \overline{B(\vec a,\delta)}}\left( I^+(\vec  x)-\lambda \textcolor{black}{\langle\vec q,\vec x\rangle}\right)\ge 0$. Note that
					$$
					\min\limits_{\vec x\in \overline{B(\vec a,\delta)}}\left( I^+(\vec  x)-\lambda \textcolor{black}{\langle\vec q,\vec x\rangle}\right)\le I^+(\vec  a)-\lambda\textcolor{black}{\langle\vec q,\vec x\rangle} =0,
					$$
					that is, \textcolor{black}{$\vec a$} is a local minimizer of the function $I^+(\vec  x)-\lambda \textcolor{black}{\langle\vec q,\vec x\rangle}$ in $B(\vec a,\delta)$. Namely, \textcolor{black}{$\vec a$} is a critical point of $I^+(\vec  x)-\lambda \textcolor{black}{\langle\vec q,\vec x\rangle}$, and then
					$$\vec 0\in\partial ( I^+(\vec  x)-\lambda\textcolor{black}{\langle\vec q,\vec x\rangle}) |_{\vec x=\vec a} =\partial I^+(\vec  a)-\lambda \vec q.$$
					Therefore, there \textcolor{black}{exist sub-differential components} $z_{ij}\in \sgn(a_i+a_j)$ and $q_i\in \sgn( a_i)$ such that $\sum_{\textcolor{black}{j:\{j,i\}\in E}}z_{ij}=\lambda d_iq_i$, $i=1,2,\cdots,n$, which means that $\vec a$ is an eigenvector of $\Delta^+_1$ and $\lambda$ is the corresponding eigenvalue. This is a contradiction.
					
					(2) $\Rightarrow$ (3):
					
					\textcolor{black}{Suppose that $\vec x_0$ satisfies {\color{black}Equation~}\eqref{eq:minsolve}, which implies}
					$$\left( I^+(\vec  x_0)-\lambda \textcolor{black}{\langle\vec q,\vec x_0\rangle}\right)\textcolor{black}{=} \min\limits_{\vec x\in \overline{B(\vec a,\delta)}}\left( I^+(\vec  x)-\lambda \textcolor{black}{\langle\vec q,\vec x\rangle}\right)<0.$$
					Since
					$$\textcolor{black}{\langle\vec q,\vec x_0\rangle}\in \sum^n_{i=1} d_i \sgn(a_i) (\vec x_0)_i\le \sum^n_{i=1} d_i |(\vec x_0)_i|=\|\vec x_0\|,$$
					it follows
					$$ I^+(\vec  x_0)-\lambda \|\vec x_0 \|\le I^+(\vec  x_0)-\lambda \textcolor{black}{\langle\vec q,\vec x_0\rangle}< 0,$$
					i.e., $I^+(\vec x_0/\|\vec x_0\|)<\lambda=I^+(\vec a)$. 
				\end{proof}
				
				\begin{remark}
					Let $G=(V,E)$ be the standard path graph on two vertices, i.e., $V=\{1,2\}$ and $E=\{\{1,2\}\}$. Then $I^+(\vec x)=|x_1+x_2|$ and $\|\vec x\|=|x_1|+|x_2|$. Let $\vec a=(1,0)\in X$ and $\vec q=(1,0)\in\partial \|\vec a\|$. Then $\vec a$ is an eigenvector corresponding to the largest $\Delta_1^+$-eigenvalue $\lambda=1$. That is, the claim (1) of Theorem \ref{th:pre_IP} doesn't hold.		
					Let $\vec a^\epsilon=(1-\epsilon,-\epsilon)$ with sufficiently small $\epsilon \in (0, \delta/2)$. 
					Then we have $I^+(\vec a^\epsilon)-\lambda\langle\vec q,\vec a^\epsilon\rangle=-\epsilon<0=I^+(\vec a)-\lambda\langle\vec q,\vec a\rangle$,
					thereby implying that  both claims (2) and (3) hold.  
				\end{remark}

				Furthermore, it is evident from the piecewise linearity of $I^+(\cdot)$ and $\langle\vec q,\cdot \rangle$ that
				\begin{equation}
					\label{eq:x_norm2}
					\begin{aligned}
						\min\limits_{\|\vec x\|_2\leq 1}\left(I^+(\vec x)-\lambda\langle\vec q,\vec x\rangle\right)&\leq I^+\left(\frac{\vec x_0}{\|\vec x_0\|_2}\right)-\lambda\langle\vec q,\frac{\vec x_0}{\|\vec x_0\|_2}\rangle\\
						&=\frac{1}{\|\vec x_0\|_2}\left( I^+\left(\vec x_0\right)-\lambda\langle\vec q,\vec x_0\rangle\right)<0
					\end{aligned}
				\end{equation}
				with a finite value of $\|\vec x_0\|_2\leq\|\vec a\|_2+\delta\leq \sum_{i\in V}1/d_i^2+\delta$, where $\vec x_0$ is selected from {\color{black}Equation~}\eqref{eq:minsolve}.
				
				According to {\color{black}Equation~}\eqref{eq:minsolve} in Theorem \ref{th:pre_IP} or {\color{black}Equation~}\eqref{eq:x_norm2}, we are able to directly propose an inverse power algorithm to approximate the smallest $\Delta^+_1$-eigenvalue, i.e., the dual Cheeger constant from {\color{black}Equation~}\eqref{eq:modified-min}.  
				We use the short name {\rm{\bf{IP}}} to denote the resulting algorithm the skeleton and local convergence of which are shown in Algorithm \ref{alg:IP_J} and 
				Theorem~\ref{th:cd1convergence}, respectively.  
				{\color{black} We remark here that the nonlinear version of  the inverse power method was first studied by Hein and B\"uhler (see Section 3 in \cite{HeinBuhler2010}).  We continue their naming of such methods, and thus  refer to Algorithm  \ref{alg:IP_J} as the inverse power algorithm and  abbreviate it as {\rm{\bf{IP}}}. 
				}

				\begin{algorithm}[t]
					\caption{\small The inverse power algorithm for approximating $\mu_1^+ = 1-h^+(G)$.}
					\label{alg:IP_J}
					\begin{algorithmic}[1]
						\State{\textbf{Input:} $\vec x^0\in X$ and $\lambda^0= I^+(\vec x^0)$.}
						\State{\textbf{Output:} the eigenvalue $\lambda^{k+1}$ and the eigenvector $\vec x^{k+1}$.}
						\State{Set $k=0$}
						\While{
							$\frac{|\lambda^{k+1}-\lambda^k|}{|\lambda^k|}\ge\varepsilon$}
						\State{
							Choose $\vec v^{k}\in\partial\|\vec x^{k}\|$}
						
						\State{Solve $\vec x^{k+1} = \argmin\limits_{\|\vec x\|_2\leq 1} I^+(\vec x) -
							\textcolor{black}{\lambda^k}\textcolor{black}{\langle\vec v^k,\vec x\rangle}$}
						
						\State{Set $\vec x^{k+1}=\frac{\vec x^{k+1}}{\|\vec x^{k+1}\|}$}
						
						\State{Set $\lambda^{k+1}=I^+(\vec x^{k+1})$}
						
						\State{Set $k \gets k+1$}
						\EndWhile
					\end{algorithmic}
				\end{algorithm}

				\begin{theorem}\label{th:cd1convergence}
					The sequence $\{I^+(\vec x^k)\}$ produced by {\rm{\bf{IP}}} is decreasingly convergent to an eigenvalue of $\Delta_1^+$. Furthermore, the sequence $\{\vec x^k\}$ produced by {\rm{\bf{IP}}} converges to an eigenvector of $\Delta^+_1$ with eigenvalue $\lambda^*\in [1-h^+(G), I^+(\vec x^0)]$.
				\end{theorem}

				\begin{proof}
					It can be easily shown that $\mu^+_1\le I^+(\vec x^{k+1})\le I^+(\vec x^k)$. So there exists $\lambda^*\in [\mu^+_1, I^+(\vec x^0)]$ such that $\{I^+(\vec x^k)\}$ converges to $\lambda^*$ decreasingly. Next we prove that $\lambda^*$ is a $\Delta^+_1$-eigenvalue.

					Denote
					\[g(\lambda,\vec v)=\min\limits_{\|\vec x\|_2\le1} I^+(\vec x)-\lambda\textcolor{black}{\langle\vec v,\vec x\rangle}.\]
					It is easy to see that $g(\lambda,\vec v)$ is uniformly continuous on $[0,1]\times \prod_{i=1}^n[-d_i,d_i]$ since $\{\vec x\in \mathbb{R}^n:\|\vec x\|_2\le1\}$ and $[0,1]\times \prod_{i=1}^n[-d_i,d_i]$ are compact.
					
					It follows from $\vec x^k\in X$ and $X$ is compact that there exist $\vec x^*\in X$ and a subsequence $\{\vec x^{k_i}\}$ of $\{\vec x^k\}$ such that $\lim_{i\to+\infty}\vec x^{k_i}=\vec x^*$. For simplicity, We may assume without loss of generality that $\{\vec x^{k_i}\}=\{\vec x^k\}$, that is, $\lim_{k\to+\infty}\vec x^k=\vec x^*$.
					According to the upper semi-continuity of the subdifferential, $\forall\,\epsilon>0, \,\exists\, \delta>0$ such that
					$\partial \|\vec x\|\subset (\partial \|\vec x^*\|)_\epsilon,$ the $\epsilon$ neighborhood of the subset $\partial \|\vec x^*\|$, $\forall\,\vec x\in B(\vec x^*,\delta)$. So there exists $N>0$ such that $\partial \|\vec x^k\|\subset (\partial \|\vec x^*\|)_\epsilon$ whenever $k> N$. Thus, $\vec v^k\in \partial (\|\vec x^*\|)_\epsilon$ for any $k > N$, which means that there is a convergent subsequence of $\{\vec v^k\}$ (note that $\partial \|\vec x^*\|$ is compact), still denoted it by $\{\vec v^k\}$ for simplicity. Then $\vec v^k\to \vec v^*$ for some $\vec v^*\in \partial \|\vec x^*\|$. Hence, according to the continuity of $g$, we have $$g(\lambda^*,\vec v^*)=\lim\limits_{k\to+\infty}g(\lambda^k,\vec v^k).$$
					
					Note that
					\[I^+(\vec x^*)=\lim\limits_{k\to+\infty} I^+(\vec x^k)=\lim\limits_{k\to+\infty}\lambda^k=\lambda^*.\]
					Suppose the contrary, that $\lambda^*$ is not an eigenvalue, then $\vec x^*$ is not an eigenvector and so by Theorem \ref{th:pre_IP} \textcolor{black}{and {\color{black}Equation~}\eqref{eq:x_norm2}}, we have
					\[
					g(\lambda^*,\vec v^*)=\min\limits_{\|\vec x\|_2\le1}I^+(\vec x)-\textcolor{black}{\lambda^*\langle\vec v^*,\vec x\rangle}<0,
					\]
					which implies that $g(\lambda^k,\vec v^k)<-\epsilon^*$
					for sufficiently large $k$ and some $\epsilon^*>0$. Therefore
					\[
					I^+(\vec x^{k+1})-\lambda^k\|\vec x^{k+1}\|\le I^+ (\vec x^{k+1})-\lambda^k\textcolor{black}{\langle\vec v^k,\vec x^{k+1}\rangle}=g(\lambda^k,\vec v^k)<-\epsilon^*,
					\]
					and then
					$$
					\lambda^{k+1}-\lambda^k=\frac{ I^+(\vec x^{k+1})}{\|\vec x^{k+1}\|}-\lambda^k<-\epsilon^*/M,
					$$
					which follows that $\lim_{k\to+\infty}\lambda^k=-\infty$, where $M$ is a given positive constant satisfying $\|\vec x\|/M\le \|\vec x\|_2\le \|\vec x\|$. This is a contradiction and then we have finished the proof.
				\end{proof}

				Up to now, {\rm{\bf{IP}}} is the first and unique algorithm for solving the dual Cheeger problem to the best of our knowledge. 	
				The feasibility of {\rm{\bf{IP}}} reflects intrinsically the local linearity of $I^+(\vec x)$ along a given direction, and 
				the same local analysis is applicable to the Cheeger problem \cite{ChangShaoZhang2015},
				as well as to the modified dual Cheeger problem \eqref{eq:modified-min}. 
				We use the short name $\widehat{\rm \bf{IP}}$ to denote the inverse power algorithm for approximating $\widehat{\mu}^+_1 = 1- \widehat{h}^+(G)$,
				which is similar to {\rm{\bf{IP}}} shown in Algorithm \ref{alg:IP_J} except for the replacement of the $5$-th and $8$-th steps with $\vec v^k\in\partial I(\vec x^k)$ and $\lambda^{k+1}=I^+(\vec x^{k+1})/(I^+(\vec x^{k+1})+I(\vec x^{k+1}))$, respectively. Here the selection of sub-differential is in the same spirit that the simple iterative algorithm solves the maxcut problem \cite{CSZZ-maxcut18}. In addition, the convergence analysis of $\widehat{\rm \bf{IP}}$ is analogous to that of {\rm{\bf{IP}}}, and the following theorem establishes the assertion.

				



				\begin{theorem}\label{th:modified-cd1convergence}
					
					The sequence $\{I^+(\vec x^k)/\left(I^+(\vec x^k)+I(\vec x^k)\right)\}$ generated by $\widehat{\rm \bf{IP}}$ decreasingly converges to an eigenvalue of $\widehat{\Delta}_1^+$.  Moreover, the sequence $\{\vec x^k\}$ produced by $\widehat{\rm \bf{IP}}$ converges to a $\widehat{\Delta}_1^+$-eigenvector with the corresponding eigenvalue $\lambda^*\in [1-\widehat{h}^+(G), I^+(\vec x^0)/\left(I^+(\vec x^0)+I(\vec x^0)\right)]$.
				\end{theorem}
				
				{\color{black}
					\begin{remark}We  must emphasise that Algorithm  \ref{alg:IP_J} does not require the design of  additional rounding strategies. 
						In fact, Theorem \ref{th:cd1convergence} states that Algorithm  \ref{alg:IP_J} ensures convergence to an eigenpair $(\mu,\vec x)$ of $\Delta_1^+$-eigenproblem, and Lemma \ref{lem:eigen1} further indicates that the eigenvector $\vec x$ shares the same eigenvalue $\mu$ with a ternary vector $\vec y$ defined by
						$$y_i=\left\{
						\begin{aligned}
							& 1,&\text{if }x_i>0,\\
							& 0, &\text{if }x_i=0,\\
							&-1,&\text{if }x_i<0,
						\end{aligned}
						\right.\;\; i=1,\cdots,n,$$
						which directly corresponds to a three cut, eliminating the need for designing additional rounding techniques. Similarly, the rounding-free property for the modified dual Cheeger constant follows from the combination of Theorem \ref{th:modified-cd1convergence} and Lemma \ref{lem:eigen1}. 
					\end{remark}	
				}
				
				\section{Recursive spectral cut algorithms for maxcut}
				\label{sec:maxcut}

				\begin{figure}[htbp]
					\centering
					\begin{tikzpicture}[>=triangle 60]
						\matrix[matrix of math nodes,column sep={42pt,between origins},row
						sep={42pt,between origins},nodes={asymmetrical rectangle}] (s)
						{
							& & &|[name=space1]| & & & &|[name=space3]|\\
							|[name=02]| G_t &|[name=A']| \vec x^{[t]} & &|[name=B']| \vec y^{[t]} & & &|[name=lt]| L_t & |[name=03]| G_{t+1} & |[name=rt]| R_t\\
							& & &|[name=space2]|& & & &|[name=space4]|\\
						};
						\draw[->] (02) edge (A')
						(A') edge node[anchor=south] {\(\rm \bf{2TSC} ~ (\bf{IP})\)} (B')  
						(A') edge node[anchor=north] {\(\rm \widehat{\bf{2TSC}} ~ (\widehat{\bf{IP}})\)} (B')      
						(space3.center) edge (lt)
						(space3.center) edge (rt)
						(space3.center) edge (03)
						(space4.center) edge (lt)
						(space4.center) --++ (rt)
						;
						\draw[-]  (B') edge (space1.center) 
						(space1.center) edge node[auto] {\(\widehat{r} > \frac{1}{2}:G_{t+1}\leftarrow G_t[V_{t+1}]\)}(space3.center)
						(B') edge (space2.center) 
						(space2.center) edgenode[auto] {\(\widehat{r} \leq \frac{1}{2}:\text{greedy algorithm}\)} (space4.center)
						;          
					\end{tikzpicture}
					\caption{Flowchart of the $t$-th RSC iteration for maxcut. Here $\vec x^{[t]}$ is the maximal eigenvector of the graph Laplacian $\Delta_2(G_t)$ on the residual graph $G_t=(V_t, E_t)$,  $\vec y^{[t]}\in \{-1,0,1\}^{|V_t|}$ gives a ternary vector, and $\widehat{r}$ is the recursive stopping indicator defined in {\color{black}Equation~}\eqref{ratio:modified_dual_cheeger}. The original RSC adopts the rounding procedure, {\rm \textbf{2TSC}} \cite{Trevisan2012} or {\rm $\widehat{\textbf{2TSC}}$} \cite{Soto2015}, which accepts $\vec x^{[t]}$ as input and outputs $\vec y^{[t]}$. In order to develop an enhanced RSC, this work replaces {\rm \textbf{2TSC}} (resp.~{\rm $\widehat{\textbf{2TSC}}$}) with {\rm{\bf{IP}}} (resp.~$\widehat{\rm \bf{IP}}$). When RSC stops at the $N$-iteration,  
						a series of partitions, $(L_1,R_1)$, $\ldots$, $(L_N,R_N)$, are obtained and can be used to form an approximate solution for maxcut. }
					\label{fig:rsc-t}
				\end{figure}
				

				%



				Given $G=(V,E)$, the dual Cheeger constant $h^+(G)$ in {\color{black}Equation~}\eqref{eq:dual-Cheeger-constant} can be regarded as the ``three'' cut version for  
				$h_{\max}(G)$ in {\color{black}Equation~}\eqref{eq:maxcut0}, where the existence of the ``uncut part'' $(V_1\cup V_2)^c$ may provide a big opportunity for solving the maxcut  in a recursive manner. This can also been seen by comparing {\color{black}Equation~}\eqref{eq:dual-Cheeger-constant} with {\color{black}Equation~}\eqref{eq:maxcut0} which tells that 
				$h^+(G)$ and $h_{\max}(G)$ share the same objective function,  but exploit different feasible domains:  $\{-1,0,1\}^{|V|}\backslash \{\vec 0\}$ for the former and
				$\{-1,1\}^{|V|}$ for the latter. Indeed, Trevisan seized this opportunity and proposed the recursive spectral cut (RSC) algorithm for maxcut \cite{Trevisan2012}. 
				Considering the $t$-th iteration as shown in {\color{black}Figure~}\ref{fig:rsc-t}, let $G_t=(V_t,E_t)$ be the residual graph, $|E_t|$ the number of edges, and $\vol_t$  the volume function acting on subset of $V_t$ with respect to $G_t$. It can be readily found that any ternary vector $\vec y^{[t]}\in \{-1,0,1\}^{|V_t|}$ naturally yields a ternary cut $(L_t,R_t,V_{t+1})$ with $L_t=\{i\in V_t:\,y_i^{[t]}=1\}$ and $R_t=\{i\in V_t:\,y_i^{[t]}=-1\}$ 
				being the sets of partitioned vertices, and $V_{t+1}=\{i\in V_t:\,y_i^{[t]}=0\}$ the set of unpartitioned ones.  After restricting the piecewise linear functions $I^+(\cdot)$, $I(\cdot)$ and $\|\cdot\|$ on the subgraph $G_t$, we have
				\begin{align}
					C^{[t]}: &= |E_t(L_t,R_t)|=\frac{1}{2}\left(\|\vec y^{[t]}\|-I^+(\vec y^{[t]})\right),\label{eq:ct} \\
					M^{[t]}: &= \frac{1}{2}\left(\vol_t(L_t\cup R_t)+|E_t(L_t\cup R_t,V_{t+1})|\right)=\frac{1}{2}\left(I^+(\vec y^{[t]})+I(\vec y^{[t]})\right),\label{eq:mt} \\
					X^{[t]}: &= |E_t(L_t\cup R_t,V_{t+1})|=I^+(\vec y^{[t]})+I(\vec y^{[t]})-\|\vec y^{[t]}\|,\label{eq:xt}\\
					r\left(\vec	y^{[t]}\right):&=\frac{C^{[t]}}{M^{[t]}-\frac{1}{2}X^{[t]}}=1-\frac{I^+(\vec y^{[t]})}{\|\vec y^{[t]}\|},\label{ratio:dual_cheeger}\\
					\widehat{r}\left(\vec y^{[t]}\right):&=\frac{C^{[t]}+\frac{1}{2}X^{[t]}}{M^{[t]}}=1-\frac{I^+(\vec y^{[t]})}{I^+(\vec y^{[t]})+I(\vec y^{[t]})}.\label{ratio:modified_dual_cheeger}
				\end{align}
				
				The remaining task is to determine the ternary vector $\vec y^{[t]}$ and it is natural to ask such $\vec y^{[t]}$ to solve the dual Cheeger problem $h^+(G_t)$ on the residual graph $G_t$. If we could do it, we will have an optimal RSC algorithm (denoted by \textbf{OPT}-RSC) in some sense. However, this is not realistic at current stage because obtaining an optimal solution to the dual Cheeger problem is NP-hard. 
				Considering the fact that the maximal eigenvalue of the graph Laplacian is strongly connected to the dual Cheeger constant and maxcut (see Theorem~\ref{th:1-ieq} and Remark~\ref{remark:maxcut}), 
				Trevisan adopted $\vec x^{[t]}$, the maximal eigenvector of the graph Laplacian $\Delta_2(G_t)$, as initial data, and then proposed 
				a rounding procedure named 2-Thresholds Spectral Cut (\textbf{2TSC}) to reach an approximate solution of $h^+(G_t)$ \cite{Trevisan2012}. 
				The procedure of \textbf{2TSC} can be briefed as follows.  For any threshold  $\tau\in [0,\max_{i\in V_t}(x_i^{[t]})^2]$, a $\tau$-related ternary vector $\vec y^{[t]}$ is generated with each component $y_i^{[t]}$ being $-1$ in the case of $x_i^{[t]}<-\sqrt{\tau}$, $0$ in the case of $|x_i^{[t]}|\leq\sqrt{\tau}$ and $1$ in the case of $x_i^{[t]}>\sqrt{\tau}$;  
				after that \textbf{2TSC} outputs a ternary vector, still denoted by $\vec y^{[t]}$, that maximizes the objective function $r$ (see {\color{black}Equation~}\eqref{ratio:dual_cheeger}) in the set of all $\tau$-related vectors, whose size is no more than $|V_t|$. We use \textbf{2TSC}-RSC to denote the resulting RSC algorithm. 
				In order to evaluate the quality of $\vec y^{[t]}$, a recursive stopping indicator, denoted by $\widehat{r}$ defined in {\color{black}Equation~}\eqref{ratio:modified_dual_cheeger}, is employed (see {\color{black}Figure~}\ref{fig:rsc-t}): RSC continues to the next iteration on the induced graph $G_{t+1}\leftarrow G_t[V_{t+1}]$ if $\widehat{r}(\vec y^{[t]})>\frac{1}{2}$; otherwise it replaces the ternary cut $(L_t,R_t,V_{t+1})$ with a binary cut $(L_t,R_t)$ that cuts more than $\frac{1}{2} |E_t|$ edges by a greedy algorithm
				and then terminates the recursion. In a word, RSC produces a series of ``good partitions'' $(L_1,R_1)$, $\ldots$, $(L_N,R_N)$ with $N$ being the number of iterations, and combines them to form an approximate solution for maxcut in a greedy manner.

				The reason of choosing $\widehat{r}$ in {\color{black}Equation~}\eqref{ratio:modified_dual_cheeger} is that the number of removing edges are $M^{[t]}$ (see {\color{black}Equation~}\eqref{eq:mt}) at the $t$-th iteration if $G_t\rightarrow G_{t+1}$, and at least the number of $C^{[t]}+\frac{1}{2}X^{[t]}$ (see {\color{black}Equations~}\eqref{eq:ct} and \eqref{eq:xt}) edges will be included in the final cut, thereby implying that a larger value of $\widehat{r}$ suggests a greater chance for obtaining an improved approximate solution for maxcut. A subtle deviation emerges accordingly in \textbf{2TSC}-RSC which adopts a different objective function $r$ (see {\color{black}Equation~}\eqref{ratio:dual_cheeger}) in the partition phase. In fact, the maximization of $\widehat{r}$ and $r$ in $\{-1,0,1\}^{|V_t|}\backslash \{\vec 0\}$ gives nothing but $\widehat{h}^+(G_t)$ and $h^+(G_t)$, respectively, as well as that in $\{-1,1\}^{|V_t|}$ yields $h_{\max}(G_t)$. Soto corrected such deviation and used the same objective function $\widehat{r}$ in both partition and evaluation phases, thereby resulting into a better lower bound of the worst-case approximation ratio for maxcut  than Trevisan's \cite{Soto2015}. Similarly, we use $\widehat{\textbf{OPT}}$-RSC to denote the extreme case the optimal solution of $\widehat{h}^+(G_t)$ is used and $\widehat{\textbf{2TSC}}$-RSC for the case the approximate solution of $\widehat{h}^+(G_t)$ is obtained through the 2-Thresholds rounding procedure. 
				
				Our main contribution to the RSC algorithm for maxcut is that we replace $\widehat{\textbf{2TSC}}$ (resp.~\textbf{2TSC}) with the proposed $\widehat{\textbf{IP}}$  (resp.~\textbf{IP}) that generates a ternary eigenvector of $\widehat{\Delta}_1^+$ (resp.~$\Delta_1^+$) that approximates $\widehat{h}^+(G_t)$ (resp.~$h^+(G_t)$) as shown in Theorems~\ref{th:1-eigen}, \ref{th:cd1convergence} and \ref{th:modified-cd1convergence}. The resulting algorithms are named $\widehat{\textbf{IP}}$-RSC and  \textbf{IP}-RSC, respectively.

				\subsection{Numerical results}

				\begin{table}[htbp]
					\centering
					\caption{Approximate values to $h^+(G)$ (resp.~$\widehat{h}^+(G)$) achieved by {\rm \textbf{2TSC}} and {\rm \textbf{IP}} (resp.~{\rm $\widehat{\textbf{2TSC}}$} and {\rm $\widehat{\textbf{IP}}$}) both of which start from the same maximal eigenvector of the graph Laplacian.}
					\label{tab:dc-mdc}
						\setlength{\tabcolsep}{4mm}{
							\begin{tabular}{c|c|c|cc||cc}
								\hline\hline
								\multirow{4}{*}{Graph} & \multirow{4}{*}{$|V|$} & \multirow{4}{*}{$|E|$} & \multicolumn{2}{c||}{\multirow{2}{*}{$h^+(G)$}}                   & \multicolumn{2}{c}{\multirow{2}{*}{$\widehat{h}^+(G)$}}                             \\
								&                        &                        & \multicolumn{2}{c||}{}                                            & \multicolumn{2}{c}{}                                                                \\ \cline{4-7} 
								&                        &                        & \multicolumn{1}{c|}{\multirow{2}{*}{\textcolor{black}{\textbf{2TSC}}}} & \multirow{2}{*}{\textbf{IP}} & \multicolumn{1}{c|}{\multirow{2}{*}{\textcolor{black}{$\widehat{\textbf{2TSC}}$}}} & \multirow{2}{*}{$\widehat{\textbf{IP}}$} \\
								&                        &                        & \multicolumn{1}{c|}{}                      &                     & \multicolumn{1}{c|}{}                      &                                        \\ \hline
								G1                     & 800                    & 19176                  & \multicolumn{1}{c|}{0.5852}                & 0.5852              & \multicolumn{1}{c|}{0.5856}                & 0.5998                                 \\ \hline
								G2                     & 800                    & 19176                  & \multicolumn{1}{c|}{0.5884}                & 0.5884              & \multicolumn{1}{c|}{0.5889}                & 0.5987                                 \\ \hline
								G3                     & 800                    & 19176                  & \multicolumn{1}{c|}{0.5892}                & 0.5892              & \multicolumn{1}{c|}{0.5903}                & 0.6026                                 \\ \hline
								G4                     & 800                    & 19176                  & \multicolumn{1}{c|}{0.5881}                & 0.5881              & \multicolumn{1}{c|}{0.5891}                & 0.6023                                 \\ \hline
								G5                     & 800                    & 19176                  & \multicolumn{1}{c|}{0.5929}                & 0.5929              & \multicolumn{1}{c|}{0.5932}                & 0.6040                                 \\ \hline
								G14                    & 800                    & 4694                   & \multicolumn{1}{c|}{0.6155}                & 0.6161              & \multicolumn{1}{c|}{0.6155}                & 0.6349                                 \\ \hline
								G15                    & 800                    & 4661                   & \multicolumn{1}{c|}{0.5945}                & 0.5988              & \multicolumn{1}{c|}{0.6001}                & 0.6325                                 \\ \hline
								G16                    & 800                    & 4672                   & \multicolumn{1}{c|}{0.6081}                & 0.6128              & \multicolumn{1}{c|}{0.6088}                & 0.6374                                 \\ \hline
								G17                    & 800                    & 4667                   & \multicolumn{1}{c|}{0.6141}                & 0.6186              & \multicolumn{1}{c|}{0.6154}                & 0.6390                                 \\ \hline
								G22                    & 2000                   & 19990                  & \multicolumn{1}{c|}{0.6441}                & 0.6441              & \multicolumn{1}{c|}{0.6453}                & 0.6611                                 \\ \hline
								G23                    & 2000                   & 19990                  & \multicolumn{1}{c|}{0.6412}                & 0.6412              & \multicolumn{1}{c|}{0.6417}                & 0.6609                                 \\ \hline
								G24                    & 2000                   & 19990                  & \multicolumn{1}{c|}{0.6416}                & 0.6416              & \multicolumn{1}{c|}{0.6419}                & 0.6588                                 \\ \hline
								G25                    & 2000                   & 19990                  & \multicolumn{1}{c|}{0.6394}                & 0.6394              & \multicolumn{1}{c|}{0.6400}                & 0.6582                                 \\ \hline
								G26                    & 2000                   & 19990                  & \multicolumn{1}{c|}{0.6379}                & 0.6387              & \multicolumn{1}{c|}{0.6396}                & 0.6577                                 \\ \hline
								G35                    & 2000                   & 11778                  & \multicolumn{1}{c|}{0.6108}                & 0.6148              & \multicolumn{1}{c|}{0.6114}                & 0.6363                                 \\ \hline
								G36                    & 2000                   & 11766                  & \multicolumn{1}{c|}{0.6055}                & 0.6070              & \multicolumn{1}{c|}{0.6081}                & 0.6351                                 \\ \hline
								G37                    & 2000                   & 11785                  & \multicolumn{1}{c|}{0.6077}                & 0.6106              & \multicolumn{1}{c|}{0.6080}                & 0.6353                                 \\ \hline
								G38                    & 2000                   & 11779                  & \multicolumn{1}{c|}{0.6046}                & 0.6100              & \multicolumn{1}{c|}{0.6053}                & 0.6347                                 \\ \hline
								G43                    & 1000                   & 9990                   & \multicolumn{1}{c|}{0.6401}                & 0.6416              & \multicolumn{1}{c|}{0.6410}                & 0.6594                                 \\ \hline
								G44                    & 1000                   & 9990                   & \multicolumn{1}{c|}{0.6445}                & 0.6445              & \multicolumn{1}{c|}{0.6445}                & 0.6603                                 \\ \hline
								G45                    & 1000                   & 9990                   & \multicolumn{1}{c|}{0.6370}                & 0.6378              & \multicolumn{1}{c|}{0.6379}                & 0.6588                                 \\ \hline
								G46                    & 1000                   & 9990                   & \multicolumn{1}{c|}{0.6395}                & 0.6395              & \multicolumn{1}{c|}{0.6397}                & 0.6567                                 \\ \hline
								G47                    & 1000                   & 9990                   & \multicolumn{1}{c|}{0.6359}                & 0.6359              & \multicolumn{1}{c|}{0.6370}                & 0.6570                                 \\ \hline
								G51                    & 1000                   & 5909                   & \multicolumn{1}{c|}{0.6169}                & 0.6174              & \multicolumn{1}{c|}{0.6174}                & 0.6350                                 \\ \hline
								G52                    & 1000                   & 5916                   & \multicolumn{1}{c|}{0.6161}                & 0.6195              & \multicolumn{1}{c|}{0.6166}                & 0.6403                                 \\ \hline
								G53                    & 1000                   & 5914                   & \multicolumn{1}{c|}{0.6139}                & 0.6172              & \multicolumn{1}{c|}{0.6141}                & 0.6351                                 \\ \hline
								G54                    & 1000                   & 5916                   & \multicolumn{1}{c|}{0.6178}                & 0.6192              & \multicolumn{1}{c|}{0.6185}                & 0.6386                                 \\ \hline\hline
							\end{tabular}
						}
					\end{table}

							\begin{table}[htbp]
								\centering
								\caption{Approximate values to $h_{\max}(G)$ achieved by {\rm \textbf{2TSC}-RSC}, {\rm \textbf{IP}-RSC}, {\rm $\widehat{\textbf{2TSC}}$-RSC} and {\rm $\widehat{\textbf{IP}}$-RSC}.}
								\label{tab:maxcut}
								\setlength{\tabcolsep}{3mm}{
										\begin{tabular}{c|c|c||c|c}
											\hline\hline
											\multirow{2}{*}{Graph} & \multirow{2}{*}{\textbf{2TSC}-RSC} & \multirow{2}{*}{\textbf{IP}-RSC} & \multirow{2}{*}{$\widehat{\textbf{2TSC}}$-RSC} & \multirow{2}{*}{$\widehat{\textbf{IP}}$-RSC} \\
											&                               &                               &                                                   &                                              \\ \hline
											G1                     & 11221                         & 11221                                             & 11262                         & 11501                                        \\ \hline
											G2                     & 11283                         & 11283                                             & 11304                         & 11481                                        \\ \hline
											G3                     & 11298                         & 11298                                             & 11339                         & 11556                                        \\ \hline
											G4                     & 11278                         & 11278                                             & 11322                         & 11548                                        \\ \hline
											G5                     & 11370                         & 11370                                             & 11378                         & 11582                                        \\ \hline
											G14                    & 2889                          & 2892                                              & 2889                          & 2980                                         \\ \hline
											G15                    & 2771                          & 2791                                              & 2840                          & 2948                                         \\ \hline
											G16                    & 2841                          & 2863                                              & 2847                          & 2979                                         \\ \hline
											G17                    & 2866                          & 2887                                              & 2894                          & 2984                                         \\ \hline
											G22                    & 12876                         & 12876                                             & 12935                         & 13215                                        \\ \hline
											G23                    & 12817                         & 12817                                             & 12945                         & 13212                                        \\ \hline
											G24                    & 12826                         & 12826                                             & 12873                         & 13170                                        \\ \hline
											G25                    & 12781                         & 12781                                             & 12830                         & 13157                                        \\ \hline
											G26                    & 12752                         & 12767                                             & 12876                         & 13148                                        \\ \hline
											G35                    & 7194                          & 7241                                              & 7219                          & 7492                                         \\ \hline
											G36                    & 7124                          & 7142                                              & 7213                          & 7473                                         \\ \hline
											G37                    & 7162                          & 7196                                              & 7259                          & 7487                                         \\ \hline
											G38                    & 7122                          & 7185                                              & 7159                          & 7476                                         \\ \hline
											G43                    & 6395                          & 6410                                              & 6430                          & 6595                                         \\ \hline
											G44                    & 6439                          & 6439                                              & 6439                          & 6596                                         \\ \hline
											G45                    & 6364                          & 6372                                              & 6402                          & 6581                                         \\ \hline
											G46                    & 6389                          & 6389                                              & 6392                          & 6560                                         \\ \hline
											G47                    & 6353                          & 6353                                              & 6401                          & 6563                                         \\ \hline
											G51                    & 3645                          & 3648                                              & 3649                          & 3752                                         \\ \hline
											G52                    & 3645                          & 3665                                              & 3675                          & 3788                                         \\ \hline
											G53                    & 3634                          & 3650                                              & 3652                          & 3758                                         \\ \hline
											G54                    & 3655                          & 3663                                              & 3705                          & 3778                                         \\ \hline\hline
										\end{tabular}
									}
							\end{table}

							We have implemented \textbf{2TSC}-RSC, \textbf{IP}-RSC, $\widehat{\textbf{2TSC}}$-RSC and $\widehat{\textbf{IP}}$-RSC and tested them on the graphs with positive weight in \textcolor{black}{G-set\footnote{Downloaded from \href{https://web.stanford.edu/~yyye/yyye/Gset/}{https://web.stanford.edu/$\sim$yyye/yyye/Gset/}}}. We set $\varepsilon=10^{-6}$ (see Algorithm \ref{alg:IP_J}) and  employ the MOSEK solver with CVX \cite{cvx} to solve the inner-subproblem in \textbf{IP} and $\widehat{\textbf{IP}}$. \textcolor{black}{It should be noted that one can design multiple algorithms for solving the inner-subproblem to optimize the overall efficiency. In this paper, we use commonly available solver CVX to ensure that the numerical results are easy to reproduce.} To the best of our knowledge, there are few numerical tests for \textbf{2TSC}-RSC or $\widehat{\textbf{2TSC}}$-RSC in the literature. 
							
							Table~\ref{tab:dc-mdc} presents the numerical values of $h^+(G)$ (resp.~$\widehat{h}^+(G)$) achieved by \textbf{2TSC} and \textbf{IP} (resp.~$\widehat{\textbf{2TSC}}$ and $\widehat{\textbf{IP}}$) both of which start from the same maximal eigenvector of the graph Laplacian. It can be readily observed there that the approximate values of $h^+(G)$ by \textbf{IP} are all larger than those by \textbf{2TSC},  and $\widehat{\textbf{IP}}$ provides much better solutions on all the graphs in approximating $\widehat{h}^+(G)$ than $\widehat{\textbf{2TSC}}$ since the ratio of the approximate value generated by $\widehat{\textbf{IP}}$ over that by $\widehat{\textbf{2TSC}}$ is at least 1.017 (see the row headed by ``G2" in Table~\ref{tab:dc-mdc}). This also constitutes the main reason for us to replace $\widehat{\textbf{2TSC}}$ (resp.~\textbf{2TSC}) with $\widehat{\textbf{IP}}$  (resp.~\textbf{IP}) in the RSC algorithm for maxcut. Table~\ref{tab:maxcut} shows the numerical values of $h_{\max}(G)$ achieved by \textbf{2TSC}-RSC, \textbf{IP}-RSC, $\widehat{\textbf{2TSC}}$-RSC and $\widehat{\textbf{IP}}$-RSC. As excepted, we are able to see there that $\widehat{\textbf{IP}}$-RSC gives the best results on all the graphs among them, thereby highlighting both the importance of maintaining consistency in the objective functions between the partition phase and the evaluation phase,  and the high solution quality of $\widehat{\textbf{IP}}$ once again. 
							It should be pointed out that we have tried but not succeeded to make a theoretical justification of such significant improvements observed in numerical experiments so far. 
							
							\subsection{Lower bound of the approximation ratio}
							

							The RSC algorithm is the first to give the nontrivial lower bound of the worst-case approximation ratio for the maxcut problem. 
							It has been proved that \textbf{2TSC}-RSC may achieve a worst-case approximation ratio of at least $0.531$ \cite{Trevisan2012},
							and $\widehat{\textbf{2TSC}}$-RSC may improve it to $0.614$ \cite{Soto2015}. Next we continue to show a similar lower bound of its extreme versions, \textbf{OPT}-RSC and $\widehat{\textbf{OPT}}$-RSC.


							
							

							\begin{theorem}\label{th:epsilon-ratio}
								If {\rm \textbf{OPT}-RSC} and {\rm $\widehat{\textbf{OPT}}$-RSC} receive an input graph whose optimal maxcut value is $1-\epsilon$, then: both of them can find solutions that cut at least a $(1-\varepsilon+\varepsilon \ln 2 \varepsilon)$ fraction of edges.
							\end{theorem}
							
							\begin{proof}
								Let us prove the first assertion. By the greedy method, every graph's optimum is not less than $\frac12$. So, for given graph $G=(V,E)$, we can assume that its optimum is $1-\varepsilon$ for some $\varepsilon\in[0,\frac12]$. That is, there exists a partition $(A,B)$ such that $\frac{|E(A,B)|}{|E|}\ge 1-\varepsilon$ and then $\frac{|E(A)|+|E(B)|}{|E|}\le \varepsilon$.
								Considering the $t$-th iteration of {\rm \textbf{OPT}-RSC} or {\rm $\widehat{\textbf{OPT}}$-RSC}, let $|E_t|:=\rho_t\cdot|E|$ be the number of edges of the residual graph $G_t$ and $\vol_t$ the volume function acting on the subsets of $V_t$ based on $G_t$ with the obtained partition ($L_t$, $R_t$), and we have
								\begin{equation}\label{eq:cut}
									\frac{2|E_t(L_t,R_t)|+|E_t(L_t\cup R_t,(L_t\cup R_t)^c)|}{\vol_t(L_t\cup R_t)+|E_t(L_t\cup R_t,(L_t\cup R_t)^c)|}\ge h_{\max}(G_t)\ge 1-\frac{\varepsilon}{\rho_t},
								\end{equation}
								thus it cuts at least a $1-\frac{\varepsilon}{\rho_t}$ fraction of the $|E|(\rho_t-\rho_{t+1})$ edges incident on $V_t$. Both the algorithms cut at least a $\max \left\{\frac{1}{2}, 1-\frac{\varepsilon}{\rho_t}\right\}$ fraction of those edges. Therefore, we divide the remaining proof into three cases:
								\begin{enumerate}
									\item If $\rho_t \geq \rho_{t+1}>2 \varepsilon$, then
									$$
									\begin{aligned}
										|E| \cdot\left(\rho_t-\rho_{t+1}\right) \cdot\left(1-\frac{\varepsilon}{\rho_t}\right) & =|E| \cdot \int_{\rho_{t+1}}^{\rho_t}\left(1-\frac{\varepsilon}{\rho_t}\right) d r \\
										& \geq|E| \cdot \int_{\rho_{t+1}}^{\rho_t}\left(1-\frac{\varepsilon}{r}\right) d r;
									\end{aligned}
									$$
									\item If $\rho_t \geq 2 \varepsilon \geq \rho_{t+1}$, then
									$$
									\begin{aligned}
										|E| \cdot\left(\rho_t-2 \varepsilon\right) \cdot\left(1-\frac{\varepsilon}{\rho_t}\right)+|E| \cdot\left(2 \varepsilon-\rho_{t+1}\right) \cdot \frac{1}{2} &\geq|E| \cdot \int_{2 \varepsilon}^{\rho_t} \left(1-\frac{\varepsilon}{r}\right) d r\\&+|E| \cdot \int_{\rho_{t+1}}^{2 \varepsilon} \frac{1}{2} d r;
									\end{aligned}
									$$
									\item If $2 \varepsilon>\rho_t \geq \rho_{t+1}$, then
									$$
									|E| \cdot\left(\rho_t-\rho_{t+1}\right) \cdot \frac{1}{2}=|E| \cdot \int_{\rho_{t+1}}^{\rho_t} \frac{1}{2} d r.
									$$
								\end{enumerate}
								Summing those bounds above, we have that the number of edges cut by the algorithm is at least
								$$
								\begin{aligned}
									\sum_t|E| \cdot\left(\rho_t-\rho_{t+1}\right) \cdot \max \left\{\frac{1}{2}, 1-\frac{\varepsilon}{\rho_t}\right\} & \geq|E| \cdot \int_{2 \varepsilon}^1\left(1-\frac{\varepsilon}{r}\right) d r+|E| \cdot \int_0^{2 \varepsilon} \frac{1}{2} d r \\
									& =(1-\varepsilon+\varepsilon \ln 2 \varepsilon) \cdot|E|,
								\end{aligned}
								$$
								which completes the proof. 
							\end{proof}

							Although seeking the partition ($L_t$, $R_t$) to satisfy {\color{black}Equation~}\eqref{eq:cut} in polynomial time is highly improbable due to the unique game conjecture,  we may still regard \textbf{OPT}-RSC and $\widehat{\textbf{OPT}}$-RSC as the maxcut approximation ceilings of \textbf{2TSC}-RSC and $\widehat{\textbf{2TSC}}$-RSC, respectively, thus we conjecture that the lower bound of the worst-case approximation ratio for \textbf{2TSC}-RSC or $\widehat{\textbf{2TSC}}$-RSC cannot be improved greater than 0.769 for	
							$$
							\min _{0 \leq \varepsilon \leq 1 / 2} \frac{1-\varepsilon+\varepsilon \ln 2 \varepsilon}{1-\varepsilon} = 0.768039,
							$$
							while the approximation ratio of the Goemans-Williamson algorithm in the sense of average is larger than $0.87856$ \cite{GoemansWilliamson1995} and has been
							proved to be an upper bound by assuming that the unique games conjecture holds \textcolor{black}{\cite{KhotKindlerMosselODonnell2007}}.

							\section{Conclusions and discussions}
							\label{sec:con}

							This paper has extended the success that the 1-Laplacian improves the Cheeger inequality (see {\color{black}Equation~}\eqref{ieq:che0}) into the equality (see {\color{black}Equation~}\eqref{eq:cheeger}) for the Cheeger constant into the (modified) dual Cheeger constant through fully respecting the intrinsic discrete nature of the underlying combinatorial problem, 
							and resulted into the spectral theory of the (modified) signless 1-Laplacian.  With the help of the revealed  equality  relations,
							an inverse power algorithm was developed for approximating the (modified) dual Cheeger constant,
							and local convergence was established. 
							When replacing the rounding procedure with the inverse power algorithm in the recursive spectral cut (RSC) algorithm for the maxcut problem,
							the approximate solutions of much higher quality than the original one were produced on G-set.
							Finally, we proved that 
							the lower bound of the worst-case approximation ratio of RSC for maxcut cannot be improved greater than $0.769$.
							Possible future research directions include a more detailed study of the spectra to the signless 1-Laplacians, and incorporations of further concepts and problems (e.g., multiway dual Cheeger constants, max-$k$-cut) from combinatorics into this equivalent spectral framework. In addition, it is also expected to prove a good theoretical approximation ratio for the lower bound of RSC in both worse-case and average senses.
							
							%
							%
							%
							

							\begin{appendix}
							\end{appendix}

		\end{document}